\documentclass[reqno, 12pt]{amsart}
\pdfoutput=1
\makeatletter
\let\origsection=\section \def\section{\@ifstar{\origsection*}{\mysection}} 
\def\mysection{\@startsection{section}{1}\z@{.7\linespacing\@plus\linespacing}{.5\linespacing}{\normalfont\scshape\centering\S}}
\makeatother        

\usepackage{amsmath,amssymb,amsthm}
\usepackage{mathrsfs}
\usepackage{mathabx}\changenotsign
\usepackage{dsfont}

\usepackage{xcolor}
\usepackage[backref]{hyperref}
\hypersetup{
    colorlinks,
    linkcolor={red!60!black},
    citecolor={green!60!black},
    urlcolor={blue!60!black}
}

\usepackage{graphicx}

\usepackage{verbatim}

\usepackage[open,openlevel=2,atend]{bookmark}

\usepackage[abbrev,msc-links,backrefs]{amsrefs} 
\usepackage{doi}

\renewcommand{\PrintDOI}[1]{\doi{#1}}

\usepackage[T1]{fontenc}
\usepackage{lmodern}
\usepackage[babel]{microtype}
\usepackage[english]{babel}

\usepackage{geometry}
\numberwithin{equation}{section}
\numberwithin{figure}{section}

\usepackage{enumitem}

\let\polishlcross=\l
\def\l{\ifmmode\ell\else\polishlcross\fi}

\def\paragraph#1{%
  \noindent\textbf{#1.}\enspace}

\let\emptyset=\varnothing
\let\setminus=\smallsetminus

\makeatletter
\def\moverlay{\mathpalette\mov@rlay}
\def\mov@rlay#1#2{\leavevmode\vtop{   \baselineskip\z@skip \lineskiplimit-\maxdimen
   \ialign{\hfil$\m@th#1##$\hfil\cr#2\crcr}}}
\newcommand{\charfusion}[3][\mathord]{
    #1{\ifx#1\mathop\vphantom{#2}\fi
        \mathpalette\mov@rlay{#2\cr#3}
      }
    \ifx#1\mathop\expandafter\displaylimits\fi}
\makeatother

\DeclareFontFamily{U}  {MnSymbolC}{}
\DeclareSymbolFont{MnSyC}         {U}  {MnSymbolC}{m}{n}
\DeclareFontShape{U}{MnSymbolC}{m}{n}{
    <-6>  MnSymbolC5
   <6-7>  MnSymbolC6
   <7-8>  MnSymbolC7
   <8-9>  MnSymbolC8
   <9-10> MnSymbolC9
  <10-12> MnSymbolC10
  <12->   MnSymbolC12}{}
\DeclareMathSymbol{\powerset}{\mathord}{MnSyC}{180}

\let\epsilon=\varepsilon

\let\rho=\varrho
\let\theta=\vartheta

\theoremstyle{plain}
\newtheorem{thm}{Theorem}[section]
\newtheorem{theorem}[thm]{Theorem}
\newtheorem{lemma}[thm]{Lemma}
\newtheorem{corollary}[thm]{Corollary}
\newtheorem{proposition}[thm]{Proposition}
\newtheorem{problem}[thm]{Problem}
\newtheorem{conjecture}[thm]{Conjecture}
\newtheorem{question}[thm]{Question}
\newtheorem{thm-intro}{Theorem}[]
\newtheorem*{claim*}{Claim}

\theoremstyle{definition}

\newtheorem{remark}[thm]{Remark}

\newtheorem{example}{Example}
\newtheorem*{example*}{Example}

\usepackage{accents}

\let\phi=\varphi

\newcommand{\restricted}{\ensuremath{{\upharpoonright}}}

\newcommand{\no}[1]{}
\newcommand{\abs}[1]{\ensuremath{{\lvert {#1} \rvert}}}

\DeclareMathOperator{\Un}{Un}
\DeclareMathOperator{\insh}{in}
\DeclareMathOperator{\outsh}{out}

\begin{document}

\author{J. Pascal Gollin and Karl Heuer}
\address{J. Pascal Gollin, Discrete Mathematics Group, Institute for Basic Science (IBS), 55, Expo-ro, Yuseong-gu, 34126 Daejeon, Republic of Korea}
\email{\tt pascalgollin@ibs.re.kr}
\address{Karl Heuer, Institute of Software Engineering and Theoretical Computer Science, Technische Universit\"{a}t Berlin, Ernst-Reuter-Platz 7, 10587 Berlin, Germany}
\email{\tt karl.heuer@tu-berlin.de}

\title[]{On the Infinite Lucchesi-Younger Conjecture I}

\begin{abstract}
    A \emph{dicut} in a directed graph is a cut for which all of its edges are directed to a common side of the cut. 
    A famous theorem of Lucchesi and Younger states that in every finite digraph the least size of an edge set meeting every dicut equals the maximum number of disjoint dicuts in that digraph.

    In this first paper out of a series of two papers, we conjecture a version of this theorem using a more structural description of this min-max property for finite dicuts in infinite digraphs. 
    We show that this conjecture can be reduced to countable digraphs where the underlying undirected graph is $2$-connected, and we prove several special cases of the conjecture.
\end{abstract}

\maketitle

\setcounter{footnote}{1}

\section{Introduction}

In finite structural graph theory there are a lot of theorems which illustrate the dual nature of certain objects 
by relating the maximum number of disjoint objects of a certain type in a graph with the minimal size of an object of another type in that graph.
More precisely, the size of the latter object trivially bounds the number of disjoint objects of the first type existing in the graph.

This duality aspect of such packing and covering results is closely related to the duality of linear programs appearing in combinatorial optimisation.
However, a purely graph theoretic interpretation requires integral solutions of both linear programs, which are hard to detect, if they even exist.

Probably the most well-known example of such a min-max result is Menger's theorem for finite undirected graphs. 
It states that for any two vertex sets~${A, B}$ in a finite graph 
the maximum number of disjoint paths between~$A$ and~$B$  
equals the minimum size of a vertex set separating~$A$ from~$B$. 
In fact, there is a structural reformulation of this quantitative description of this dual nature of connectivity: 
for any two vertex sets~${A, B}$ in a finite graph 
there exists a set of disjoint paths between~$A$ and~$B$ 
together with a vertex set separating~$A$ from~$B$ 
that consists of precisely one vertex from each of the paths. 

While for finite graphs this is an easy corollary from the quantitative version, in infinite graphs it turns out that such a structural version is much more meaningful. 
While Erd{\H{o}}s observed that a version of Menger's theorem based on the equality of infinite cardinals is quite trivial, he conjectured that the analogue of the structural version is the better way to interpret this dual nature of connectivity. 
Such a version has been established by Aharoni and Berger~\cite{inf-menger}. 
Their theorem restored many of the uses of connectivity duality that the trivial cardinality version could not provide, 
and hence it influenced much of the development of infinite connectivity theory and matching theory.

Another such min-max theorem was established by Lucchesi and Younger~\cite{lucc-young_paper} for directed graphs. 
To state that theorem we have to give some definitions first. 

In a weakly connected directed graph~$D$ we call a cut of~$D$ \emph{directed}, or a \emph{dicut} of~$D$, if all of its edges have their head in a common side of the cut. 
We call a set of edges a \emph{dijoin} of~$D$ if it meets every non-empty dicut of~$D$. 
Now we can state the mentioned theorem. 

\begin{theorem}
    \label{thm:LY-fin}
    \cite{lucc-young_paper}
    In every weakly connected finite digraph, the maximum number of disjoint dicuts equals the minimum size of a dijoin.
\end{theorem}

Beside the original proof of Theorem~\ref{thm:LY-fin} due to 
Lucchesi and Younger~\cite{lucc-young_paper}, further ones appeared. 
Among them are 
an inductive proof by Lov\'asz~\cite{lovasz}*{Thm.~2} and 
an algorithmic proof of Frank~\cite{frank_book}*{Section~9.7.2}. 
As for Menger's Theorem, we now state a structural reformulation of Theorem~\ref{thm:LY-fin}, which for finite digraphs is easily seen to be equivalent.

\begin{theorem}
    \label{thm:LY-fin-structural}
    Let~$D$ be a finite weakly connected digraph. 
    Then there exists a tuple~${(F, \mathcal{B})}$ such that the following statements hold.
    \begin{enumerate}[label=(\roman*)]
        \item {${\mathcal{B}}$ is a set of disjoint dicuts of~$D$.}
        \item {${F \subseteq E(D)}$ is a dijoin of~$D$.}
        \item {${F \subseteq \bigcup \mathcal{B}}$.}
        \item {${\abs{F \cap B} = 1}$ for every~${B \in \mathcal{B}}$.}
    \end{enumerate}
\end{theorem}

In this paper 
we consider the question whether Theorem~\ref{thm:LY-fin-structural} extends to infinite digraphs. 
Let us first show that a direct extension of this formulation to arbitrary infinite digraphs fails. 
To do this we define a \emph{double ray} to be an undirected two-way infinite path.
Now consider the digraph depicted in Figure~\ref{fig:counterexample}. 
Its underlying graph is the Cartesian product of a double ray with an edge. 
Then we consistently orient all edges corresponding to one copy of the double ray in one direction and all edges of the other copy in the different direction. 
Finally, we direct all remaining edges such that they have their tail in the same copy of the double ray. 
This digraph contains no finite dicut, but it does contain infinite ones. 
Note that every dicut of this digraph contains at most one horizontal edge, which corresponds to an oriented one of some copy of the double ray, and all vertical edges to the left of some vertical edge. 
Hence, we cannot even find two disjoint dicuts. 
However, a dijoin of this digraph cannot be finite, as we can easily find a dicut avoiding any finite set of edges by considering a horizontal edge to the left of the finite set. 
So we obtain that each dijoin hits every dicut infinitely often in this digraph.
Therefore, neither the statement of Theorem~\ref{thm:LY-fin-structural} nor the statement of Theorem~\ref{thm:LY-fin} remain true if we consider arbitrary dicuts in infinite digraphs. 

\begin{figure}[htbp]
    \centering
    \includegraphics{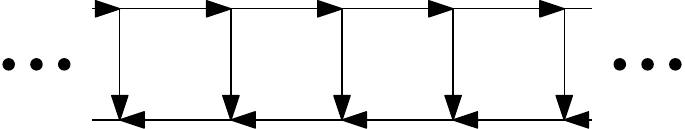}
    \caption{A counterexample to an extension of Theorem~\ref{thm:LY-fin-structural} to infinite digraphs where infinite dicuts are considered too.}
    \label{fig:counterexample}
\end{figure}

Another counterexample for these naive extensions is the infinite transitive tournament with vertex set~$\mathbb{N}$ and with an edge directed from~$m$ to~$n$ if and only if~$m$ is smaller than~$n$ for all~${m, n \in \mathbb{N}}$.
We leave the verification of this fact to the reader. 

In order to overcome the problem of this example let us again consider the situation in Menger's theorem. 
There, even in the infinite version, we are only considering finite paths for those objects that we want to pack.
Together with the example in Figure~\ref{fig:counterexample}, this suggests that we might need to restrict our attention to finite dicuts when extending Theorem~\ref{thm:LY-fin-structural} to infinite digraphs. 
Hence, we make the following definitions.

In a weakly connected digraph~$D$ we call an edge set~${F \subseteq E(D)}$ a \emph{finitary dijoin} of~$D$ if it intersects every non-empty finite dicut of~$D$. 
Building up on this definition, we call a tuple~${(F, \mathcal{B})}$ as in Theorem~\ref{thm:LY-fin-structural} but where~$F$ is now a finitary dijoin and~$\mathcal{B}$ is a set of disjoint finite dicuts of~$D$, an \emph{optimal pair} for~$D$. 

Not in contradiction to the example given above, we now state the following conjecture raised by the second author of this article, which we call the Infinite Lucchesi-Younger Conjecture.

\begin{conjecture}
    \label{conj:LY-inf}
    Every weakly connected digraph admits an optimal pair.
\end{conjecture}

Apparently, an extension of Theorem~\ref{thm:LY-fin} as in Conjecture~\ref{conj:LY-inf} turns out to be very similar to a more general problem about infinite hypergraphs independently raised by Aharoni~\cite{aharoni-conj}*{Prob.~6.7}. 
We will discuss this connection further in Section~\ref{sec:hypergraphs}.

\vspace{0.2cm}

The three mentioned proofs 
\cite{lucc-young_paper}
\cite{lovasz}*{Thm.~2} 
\cite{frank_book}*{Thm.~9.7.2} 
of Theorem~\ref{thm:LY-fin} even show a slightly stronger result.
We call an optimal pair \emph{nested} if the elements of~$\mathcal{B}$ are pairwise nested, 
i.e.~any two finite dicuts~${E(X_1, X_2), E(Y_1, Y_2) \in \mathcal{B}}$ satisfy one of the following conditions: 
${X_1 \subseteq Y_1}$, 
${Y_1 \subseteq X_1}$, 
${X_1 \subseteq Y_2}$, or 
${Y_2 \subseteq X_1}$.

\begin{theorem}
    \label{thm:LY-fin-nested}
    \cite{lucc-young_paper}
    Every weakly connected finite digraph admits a nested optimal pair.
\end{theorem}

Hence, we also make the following conjecture.

\begin{conjecture}
    \label{conj:LY-inf-nested}
    Every weakly connected digraph admits a nested optimal pair.
\end{conjecture}

In weakly connected infinite digraphs there are indications that, in contrast to the finite case, Conjecture~\ref{conj:LY-inf-nested} may be strictly stronger than Conjecture~\ref{conj:LY-inf}. 
In Section~\ref{sec:examples} we will illustrate examples of digraphs with a finitary dijoin which is part of an optimal pair, but not of any nested one. 

\vspace{0.2cm}

One of the main results of this paper is the reduction of Conjectures~\ref{conj:LY-inf} and~\ref{conj:LY-inf-nested} to countable digraphs with a certain separability property and whose underlying multigraphs are $2$-connected. 
We call a digraph~$D$ \emph{finitely diseparable} if for any two vertices~${v, w \in V(D)}$ there is a finite dicut of~$D$ such that~$v$ and~$w$ lie in different sides of that finite dicut. 

\begin{theorem}
    \label{thm:reduction}
    If Conjecture~\ref{conj:LY-inf} (or Conjecture~\ref{conj:LY-inf-nested}, respectively) holds for all countable finitely diseparable digraphs whose underlying multigraphs are $2$-connected, then Conjecture~\ref{conj:LY-inf} (or Conjecture~\ref{conj:LY-inf-nested}, respectively) holds for all weakly connected digraphs. 
\end{theorem}

Moreover, we verify Conjecture~\ref{conj:LY-inf-nested} for several classes of digraphs. 
We gather all these results in the following theorem.  
Before we can state the theorem we have to give some further definitions.
We call a minimal non-empty dicut of a digraph a \emph{dibond}. 
Furthermore, we call an undirected one-way infinite path a \emph{ray}. 
We say a digraph is \emph{rayless} if its underlying multigraph does not contain a ray.

\begin{theorem}
    \label{thm:LY-inf-cases}
    Conjecture~\ref{conj:LY-inf-nested} holds for a weakly connected digraph~$D$ if it has any of the following properties:
    \begin{enumerate}[label=(\roman*)]
        \item \label{item:LY-fin-dijoin}
            There exists a finitary dijoin of~$D$ of finite size. 
        \item \label{item:LY-fin-disj-dicuts}
            The maximal number of disjoint finite dicuts of~$D$ is finite. 
        \item \label{item:LY-fin-nested-disj-dicuts}
            The maximal number of disjoint and pairwise nested finite dicuts of~$D$ is finite. 
        \item \label{item:LY-edge-fin-dibonds}
            Every edge of~$D$ lies in only finitely many finite dibonds of~$D$. 
        \item \label{item:LY-fin-dibonds}
            $D$ has no infinite dibond. 
        \item \label{item:LY-rayless}
            $D$ is rayless. 
    \end{enumerate}
\end{theorem}

The structure of this paper is as follows.
In Section~\ref{sec:prelims} we introduce our needed notation and prove some basic tools that we will need throughout the paper. 
In Section~\ref{sec:examples} we will discuss some examples which shall illustrate the difficulties of relating Conjecture~\ref{conj:LY-inf} to Conjecture~\ref{conj:LY-inf-nested}. 
Section~\ref{sec:reduce} is dedicated to the proof of Theorem~\ref{thm:reduction}.
In Section~\ref{sec:cases} we shall deduce the items of Theorem~\ref{thm:LY-inf-cases} via several lemmas by lifting Theorem~\ref{thm:LY-fin-nested} to infinite digraphs via the compactness principle. 
Section~\ref{sec:hypergraphs} is dedicated to a short discussion of the connection between Conjecture~\ref{conj:LY-inf} and the more general problem from Aharoni about matchings in infinite hypergraphs. 

\vspace{0.2cm}

In a second paper~\cite{inf-LY:2} on the Infinite Lucchesi-Younger Conjecture we will extend several parts of the algorithmic proof of Frank~\cite{frank_book}*{Section~9.7.2} for Theorem~\ref{thm:LY-fin} to infinite digraphs. 
This proof is based on the ideas of the negative circuit method developed for more general submodular frameworks by Fujishige~\cite{fujishige} and Zimmermann~\cite{zimmermann}. 
Instead of just starting with a dijoin of minimum size, the idea of Frank's proof is to start with any dijoin and algorithmically ``improve'' it with the help of cycles of negative cost in an auxiliary digraph whose definition depends on the dijoin. 
Once the dijoin can no longer be ``improved'' some structural properties of the auxiliary graph help in fining the desired set of dibonds which together with the dijoin form a nested optimal pair.

\section{Basic notions and tools}
\label{sec:prelims}

For basic facts about finite and infinite graphs we refer the reader to \cite{diestel_buch}. 
Several proofs, especially in Section~\ref{sec:cases}, rely on the compactness principle in combinatorics.
We omit stating it here but refer to~\cite{diestel_buch}*{Appendix~A}. 
Especially for facts about directed graphs we refer to~\cite{bang-jensen}. 

In general, we allow our digraphs to have parallel edges, but no loops unless we explicitly mention them. 
Similarly, all undirected multigraphs we consider do not have loops if nothing else is explicitly stated. 

\vspace{0.2cm}

Throughout this section let~$D$ denote a digraph with vertex set~$V(D)$ and edge set~$E(D)$. 
We view the edges of~$D$ as ordered pairs~${(u,v)}$ of vertices~${u, v \in V(D)}$ and shall write~${uv}$ instead of~${(u, v)}$, although this might not uniquely determine an edge.
In parts where a finer distinction becomes important we shall clarify the situation.
For an edge~${uv \in E(D)}$ we furthermore call the vertex~$u$ as the \emph{tail} of~${uv}$ and~$v$ as the \emph{head} of~$uv$.
We denote the underlying undirected multigraph of~$D$ by~$\Un(D)$.

In an undirected non-trivial path we call the vertices incident with just one edge the \emph{endvertices} of that path.
For the trivial path consisting just of one vertex, we call that vertex also an \emph{endvertex} of that path.
If~$P$ is an undirected path with endvertices~$v$ and~$w$, we call~$P$ a \emph{$v$--$w$~path}.
For a path~$P$ containing two vertices~${x, y \in V(P)}$ we write~${xPu}$ for the $x$--$u$ subpath contained in~$P$.
Should~$P$ additionally be a directed path where~$v$ has out-degree~$1$, then we call~$P$ a \emph{directed $v$--$w$~path}.
We also allow to call the trivial path with endvertex~$v$ a \emph{directed $v$--$v$~path}.
For two vertex sets~${A, B \subseteq V(D)}$ we call an undirected path~${P \subseteq D}$ an \emph{$A$--$B$~path} if~$P$ is an $a$--$b$~path for some~${a \in A}$ and~${b \in B}$ but is disjoint from~${A \cup B}$ except from its endvertices.
Similarly, we call a directed path that is an $A$--$B$~path a \emph{directed $A$--$B$~path}.

We call an undirected graph a \emph{star} if it is isomorphic to the complete bipartite graph~$K_{1, \kappa}$ for some cardinal~$\kappa$, where the vertices of degree~$1$ are its \emph{leaves} and the vertex of degree~$\kappa$ is its \emph{centre}.

We define a \emph{ray} to be an undirected one-way infinite path. 
Any subgraph of a ray~$R$ that is itself a ray is called a \emph{tail} of~$R$. 
The unique vertex of~$R$ of degree~$1$ is the \emph{start vertex of~$R$}.

An undirected multigraph that does not contain a ray is called \emph{rayless}. 

A \emph{comb~$C$ with teeth~$U$} is an undirected graph~$C$ with subset~${U \subseteq V(C)}$ of its vertices 
such that~$C$ is the union of a ray~$R$ together with infinitely many disjoint undirected finite (possibly trivial) $U$--$V(R)$ paths. 
The ray~$R$ is called the \textit{spine} of~$C$.

The following lemma is a fundamental tool in infinite graph theory. 
A basic version of this lemma which does not take different infinite cardinalities into account can be found in~\cite{diestel_buch}*{Lemma~8.2.2}. 
We shall only apply this more general version for vertex sets of cardinality~$\aleph_0$ and~$\aleph_1$ in this paper.

\begin{lemma}
    \label{lem:star-comb}
    \cite{star-comb}*{Lemma~2.5}
    Let~$G$ be an infinite connected undirected multigraph and let~${U \subseteq V(G)}$ be such that~${\abs{U} = \kappa}$ for some infinite regular cardinal~$\kappa$.
    Then there exists a set~${U' \subseteq U}$ with~${\abs{U'} = \abs{U}}$ such that~$G$ either contains a comb with teeth~$U'$ or a subdivided star whose set of leaves is~$U'$. 
\end{lemma}

\subsection{Cuts and dicuts}

Throughout this subsection let~$D$ denote a weakly connected digraph.
For two vertex sets~${X, Y \subseteq V(D)}$ we define~${E_D(X, Y) \subseteq E(D)}$ as the set of those edges 
that have their head in~${X \setminus Y}$ and their tail in~${Y \setminus X}$, 
or their head in~${Y \setminus X}$ and their tail in~${X \setminus Y}$. 
Furthermore, we define 
\[
    {\overrightarrow{E}_D(X, Y) := \{ uv \in E(X, Y) \, | \, u \in X \textnormal{ and } v \in Y \}}.
\]
We will usually omit the subscript if the graph we are talking about is clear from the context. 

If~${X \cup Y = V(D)}$ and~${X \cap Y = \emptyset}$, we call~${E(X, Y)}$ a \emph{cut} of~$D$ and refer to~$X$ and~$Y$ as the \emph{sides} of the cut. 
Moreover, by writing~${E(M, N)}$ and calling it a cut of~$D$ we implicitly assume~$M$ and~$N$ to be the sides of that cut, 
and by calling an edge set~$B$ a cut we implicitly assume that~$B$ is of the form~${E(M,N)}$ for suitable sets~$M$ and~$N$. 

We call two cuts~${E(X_1, Y_1)}$ and~${E(X_2, Y_2)}$ of~$D$ \emph{nested} 
if one of~${X_1, Y_1}$ is $\subseteq$-comparable with one of~${X_2, Y_2}$.
Moreover, we call a set or sequence of cuts of~$D$ nested if its elements are pairwise nested. 
If two cuts of~$D$ are not nested, we call them \emph{crossing} (or say that they \emph{cross}). 

A cut is said to \emph{separate} two vertices~${v, w \in V}$ if~$v$ and~$w$ lie on different sides of that cut. 

A minimal non-empty cut is called a \emph{bond}. 
Note that a cut~${E(X, Y)}$ is a bond, if and only if the induced subdigraphs~${D[X]}$ and~${D[Y]}$ are weakly connected digraphs. 

We call a cut~${E(X, Y)}$ \emph{directed}, or briefly a \emph{dicut}, if all edges of~${E(X, Y)}$ have their head in one common side of the cut. 
A bond that is also a dicut is called a \emph{dibond}.

We call~$D$ \emph{finitely separable} if for any two different vertices~${v, w \in V}$ there exists a finite cut of~$D$ such that~$v$ and~$w$ are separated by that cut. 
Note that if two vertices are separated by some finite cut, then they are separated by some finite bond as well. 
If furthermore any two different vertices~${v, w \in V(D)}$ can even be separated by a finite dicut, or equivalently a finite dibond, of~$D$, we call~$D$ \emph{finitely diseparable}.

For a vertex set~${X \subseteq V(D)}$ we define 
\begin{itemize}
    \item ${\delta_D^-(X) := \overrightarrow{E}(V(D) \setminus X, X)}$, the set of in-going edges of~$X$;
    \item ${\delta_D^+(X) := \overrightarrow{E}(X, V(D) \setminus X)}$, the set of out-going edges of~$X$;
    \item ${\delta_D(X) := \delta^-(X) \cup \delta^+(X)}$, the set of incident edges of~$X$.
\end{itemize}
As before, we will usually omit the subscript if the graph we are talking about is clear from the context. 

Given a dicut~${B = \overrightarrow{E}(X, Y)}$ we call~$Y$ the \emph{in-shore} of~$B$ and~$X$ the \emph{out-shore} of~$B$.
We shall also write~$\insh(B)$ for the in-shore of the dicut~$B$ and~$\outsh(B)$ for the out-shore of~$B$.

For undirected multigraphs \emph{cuts}, \emph{bonds}, \emph{sides}, the notion of being \emph{nested} and the notion of \emph{separating} two vertices are analogously defined.
Hence, we call an undirected multigraph \emph{finitely separable} if any two vertices can be separated by a finite cut of the multigraph.
Furthermore, in an undirected multigraph~$G$ with~${X, Y \subseteq V(G)}$ we write~${E(X, Y)}$ for the set of those edges of~$G$ that have one endvertex in~${X \setminus Y}$ and the other in~${Y \setminus X}$.

\vspace{0.2cm}

Given a set~${\mathcal{B} = \{ B_i \, | \, i \in I\}}$ of dicuts of~$D$, we write
\begin{itemize}
    \item ${\bigwedge \mathcal{B} := \delta^{-} \left( \bigcap \{ \insh(B) \, | \, B \in \mathcal{B} \} \right)}$, 
        or simply~${B_1 \wedge B_2}$ for~${\bigwedge \{ B_1, B_2 \}}$; and
    \item ${\bigvee \mathcal{B} := \delta^{-} \left( \bigcup \{ \insh(B) \, | \, B \in \mathcal{B} \} \right)}$, 
        or simply~${B_1 \vee B_2}$ for~${\bigvee \{ B_1, B_2 \}}$.
\end{itemize}

Note that since~$D$ is weakly-connected, 
${\bigwedge \mathcal{B}}$ is empty if and only if~${\bigcap \{ \insh(B) \, | \, B \in \mathcal{B} \}}$ is empty, and 
${\bigvee \mathcal{B}}$ is empty if and only if~${\bigcup \{ \insh(B) \, | \, B \in \mathcal{B} \}}$ equals~$V(D)$. 

\begin{remark}
    \label{rem:corners1}
    Let~${\mathcal{B}}$ be a set of dicuts of~$D$.
    \begin{enumerate}
        \item ${\bigwedge \mathcal{B}}$ is a (possibly empty) dicut of~$D$.
        \item ${\bigvee \mathcal{B}}$ is a (possibly empty) dicut of~$D$. \qed
    \end{enumerate}
\end{remark}

Note that~${\bigwedge \mathcal{B}}$ and~${\bigvee \mathcal{B}}$ might be infinite dicuts of~$D$, even if each~${B \in \mathcal{B}}$ is finite. 
Furthermore, note that if~$B_1$ and~$B_2$ are dibonds then~${B_1 \wedge B_2}$ does not need to be a dibond, even if it is non-empty. 
A simple double-counting argument yields the following. 

\begin{remark}
    \label{rem:corners2}
    Let~$B_1$ and~$B_2$ be dicuts of~$D$, and let~${F \subseteq E(D)}$. 
    Then 
    \begin{enumerate}
        \item ${(B_1 \cap F) \cup (B_2 \cap F) = ((B_1 \wedge B_2) \cap F) \cup ((B_1 \vee B_2) \cap F)}$; and
        \item ${\abs{B_1 \cap F} + \abs{B_2 \cap F} = \abs{(B_1 \wedge B_2) \cap F} + \abs{(B_1 \vee B_2) \cap F}}$.
    \end{enumerate}

    Moreover, if~${B_1}$ and~${B_2}$ are disjoint, then~${B_1 \wedge B_2}$ and~${B_1 \vee B_2}$ are disjoint as well. \qed
\end{remark}

\vspace{0.2cm}

Let~$B$ be a dicut.
We call a set~${\mathcal{B} = \{ B_i \, | \, i \in I \}}$ a \emph{decomposition of~$B$} if for each~${i, j \in I}$
\begin{itemize}
    \item ${B_i \subseteq B}$ is a (possibly empty) dicut; 
    \item ${B_i \cap B_j = \emptyset}$ for~${i \neq j}$; 
    \item ${\bigcup_{k \in I} B_k = B}$.
\end{itemize}
We write~${B = \bigoplus \mathcal{B}}$ if~${\mathcal{B}}$ is a decomposition of~$B$.

\subsection{Dijoins and optimal pairs for classes of finite dibonds}

Throughout this subsection let~$D$ denote a weakly connected digraph.
We call an edge set~${F \subseteq E(D)}$ a \emph{dijoin} of~$D$ if~${F \cap B \neq \emptyset}$ holds for every 
dicut~$B$ of~$D$.
Similarly, we call an edge set~${F \subseteq E(D)}$ a \emph{finitary dijoin} of~$D$ if~${F \cap B \neq \emptyset}$ holds for every 
finite dicut~$B$ of~$D$. 
Note that an edge set~${F \subseteq E(D)}$ is already a (finitary) dijoin if~${F \cap B \neq \emptyset}$ holds for every (finite) dibond of~$D$ since every (finite) dicut is a disjoint union of (finite) dibonds.

\vspace{0.2cm}

Let~$\mathfrak{B}$ be a 
class of finite dibonds of~$D$. 
Then we call an edge set~${F \subseteq E(D)}$ a \emph{$\mathfrak{B}$-dijoin} of~$D$ 
if~${F \cap B \neq \emptyset}$ holds for every~${B \in \mathfrak{B}}$. 
Note that for the class~${\mathfrak{B}_{\textnormal{fin}}}$ of all finite dibonds of~$D$ we immediately get that the finitary dijoins of~$D$ are precisely the ${\mathfrak{B}_{\textnormal{fin}}}$-dijoins of~$D$. 

We call a tuple~${(F,\mathcal{B})}$ a \emph{$\mathfrak{B}$-optimal pair} for~$D$ if 
\begin{enumerate}[label=(\roman*)]
    \item ${F \subseteq E(D)}$ is a $\mathfrak{B}$-dijoin of~$D$;
    \item ${\mathcal{B} \subseteq \mathfrak{B}}$ is a set of disjoint dibonds in~$\mathfrak{B}$;
    \item ${F \subseteq \bigcup \mathcal{B}}$; and
    \item ${\abs{F \cap B} = 1}$ for every~${B \in \mathcal{B}}$.
\end{enumerate}

We call a $\mathfrak{B}$-optimal pair~${(F,\mathcal{B})}$ for~$D$ \emph{nested} if the elements of~$\mathcal{B}$ are pairwise nested.

Note that the (nested) optimal pairs as defined in the introduction are precisely the (nested) $\mathfrak{B}_\textnormal{fin}$-optimal pairs.

Using the introduced notation we state the following question, which is the general main topic of our studies in this paper and the second paper~\cite{inf-LY:2} of this series.

\begin{question}
    \label{q:general-LY}
    Which weakly connected digraphs and classes~$\mathfrak{B}$ of finite dibonds admit a (nested) $\mathfrak{B}$-optimal pair?
\end{question}

\noindent Note that this question is more general and flexible than Conjecture~\ref{conj:LY-inf} or Conjecture~\ref{conj:LY-inf-nested} but encompasses them by setting~${\mathfrak{B} = \mathfrak{B}_{\textnormal{fin}}}$.

\vspace{0.2cm}

Let~$\mathfrak{B}$ be a class of dicuts of~$D$. 
Then let~$\mathfrak{B}^\oplus$ denote the class of dicuts~$B$ of~$D$ 
which have a partition~${B = \bigoplus \mathcal{B}}$ for some~${\mathcal{B} \subseteq \mathfrak{B}}$.

We say~$\mathfrak{B}$ is \emph{finitely corner-closed} if 
\begin{enumerate}
    \item If~${B_1, B_2 \in \mathfrak{B}}$ then 
    ${B_1 \wedge B_2 \in \mathfrak{B}^\oplus}$.
    \item If~${B_1, B_2 \in \mathfrak{B}}$ then 
    ${B_1 \vee B_2 \in \mathcal{B}^\oplus}$.
\end{enumerate}

Note that~$(\mathfrak{B}^\oplus)^\oplus = \mathfrak{B}^\oplus$.
 
Throughout this paper we will mostly consider classes of finite dibonds of~$D$ which are finitely corner-closed, for example~$\mathfrak{B}_\textnormal{fin}$, the class of finite dibonds of~$D$.

\subsection{Finitely separable multigraphs}

In this section we prove certain size related properties of finitely separable multigraphs using Lemma~\ref{lem:star-comb}.

For a multigraph~$G$ we call a subgraph~${X \subseteq G}$ a \emph{$2$-block} of~$G$ if~$X$ is a maximal connected subgraph without a cutvertex. 
Hence a $2$-block of a connected multigraph either consists of a set of pairwise parallel edges in~$G$ or is a maximal $2$-connected subgraph of~$G$. 
In a digraph~$D$ we call a subdigraph~$X$ a \emph{$2$-block} of~$D$ if~$\Un(X)$ is a $2$-block of~$\Un(D)$.

One tool we will use in this paper is the \emph{$2$-block-cutvertex-tree} (cf.~\cite{diestel_buch}*{Lemma~3.1.4}). 
Let~$\mathcal{X}$ denote the set of all $2$-blocks in~$G$, and~$C$ the set of all cutvertices in~$G$. 
Then the bipartite graph with vertex set~${\mathcal{X} \cup C}$ with edge set~${\{ c X \, | \, c \in C, X \in \mathcal{X}, c \in X \}}$ is a tree, the \emph{$2$-block-cutvertex-tree}. 

We immediately get the following remark.

\begin{remark}
    \label{rem:bonds-in-blocks}
    Let~$G$ be a multigraph or a digraph.
    \begin{enumerate}
        [label=(\roman*)]
        \item \label{item:bonds-in-blocks1}
            Every bond of~$G$ is contained in a unique $2$-block.
        \item \label{item:bonds-in-blocks2}
            Bonds of~$G$ that are contained in different $2$-blocks are nested. 
    \end{enumerate}
\end{remark}

\begin{lemma}
    \label{lem:2-block}
    \begin{enumerate}
        [label=\textit{(\roman*)}]
        \item \label{item:2-block-countable}
            Every $2$-block of a finitely separable multigraph or digraph is countable.
        \item \label{item:2-block-rayless}
            Every $2$-block of a finitely separable rayless multigraph or digraph is finite.
    \end{enumerate}
\end{lemma}

\begin{proof}
   Let~$G$ be a finitely separable multigraph and let~$X$ be a $2$-block of~$G$. 
   Assume for a contradiction that either~$X$ is infinite and rayless, or~$X$ is uncountable. 
   Let~$U$ be a subset of~$V(X)$ with~${\abs{U} = \min \{ \abs{X}, \aleph_1 \}}$. 
    Applying Lemma~\ref{lem:star-comb} to~$U$ in~$X$, we obtain a subdivided star~$S_1$ in~$X$ whose set of leaves~$L_1$ satisfies~${\abs{L_1} = \abs{U}}$. 
    Let~$c_1$ be the centre of~$S_1$. 
    Using that~$X$ is $2$-connected, we now apply Lemma~\ref{lem:star-comb} to~$L_1$ in~${G-c_1}$, which is still connected. 
    Hence, we obtain a subdivided star~$S_2$ in~${G-c_1}$ whose set of leaves~$L_2$ satisfies~${\abs{L_2} = \abs{L_1}}$ and~${L_2 \subseteq L_1}$. 
    Let~$c_2$ denote the centre of~$S_2$.
    Now we get a contradiction to~$G$ being finitely separable because~$S_1$ and~$S_2$ have infinitely many common leaves in~$L_2$. 
    So~${G[V(S_1) \cup V(S_2)]}$ contains infinitely many internally disjoint $c_1$--$c_2$~paths, witnessing that~$c_1$ and~$c_2$ cannot be separated by a finite cut of~$G$.
    
    To complete the proof we still need to consider for a contradiction a $2$-block~$X$ of~$G$ whose vertex set is countable (in case~\ref{item:2-block-countable}) or finite (in case~\ref{item:2-block-rayless}) but whose edge set is uncountable (in case~\ref{item:2-block-countable}) or infinite (in case~\ref{item:2-block-countable}). 
    A contradiction to the fact that~$X$ is finitely separable arises by an easy application of the pigeonhole principle to the two-element subsets of~${V(X)}$. 
\end{proof}

Together with Remark~\ref{rem:bonds-in-blocks} we obtain the following immediate corollary. 

\begin{corollary}
    \label{cor:fin-sep+rayless=no_inf-bond}
    A finitely separable rayless multigraph has no infinite bond. 
    \qed
\end{corollary}

\subsection{Quotients}
\label{subsec:quotients}

Throughout this subsection let~$G$ denote a digraph or a multigraph. For a set~${N \subseteq E(G)}$ let~${G{/}N}$ denote the contraction minor of~$G$ which is obtained by contracting inside~$G$ all edges of~$N$ and deleting all loops that might occur. 
Similarly, we define~${G{.}N := G{/}(E(G) \setminus N)}$.
For a vertex~${v \in V(G)}$ and any contraction minor~${G{.}N}$ with~${N \subseteq E(G)}$ let~$\dot{v}$ denote the vertex in~${G{.}N}$ which corresponds to the contracted, possibly trivial, (weak) component of~${G - N}$ containing~$v$.

We state the following basic lemma without proof.

\begin{lemma}
    \label{lem:contr-cuts}
    Let~${B, N \subseteq E(G)}$ with~${B \subseteq N}$ and 
    let~${v, w \in V(G)}$. 
    Then~$B$ is a cut (or dicut/bond/dibond, respectively) of~$G$ that separates~$v$ and~$w$ if and only if~$B$ is a cut (or dicut/bond/dibond, respectively) of~$G{.}N$ that separates~$\dot{v}$ and~$\dot{w}$. 
    
    Moreover, two cuts~${B_1, B_2 \subseteq N}$ are nested as cuts of~$G$ if and only if they are nested as cuts of~$G{.}N$.
    \qed
\end{lemma}

Given a set~$\mathcal{B}$ of cuts of~$G$, we define an equivalence relation on~$V(G)$ 
by setting~${v \equiv_\mathcal{B} w}$ if and only if 
we cannot separate~$v$ from~$w$ by a cut in~$\mathcal{B}$. 
It is easy to check that~$\equiv_\mathcal{B}$ is indeed an equivalence relation.
For~${v \in V(G)}$ we shall write~${[v]_{\equiv_\mathcal{B}}}$ for the equivalence class with respect to~$\equiv_\mathcal{B}$ containing~$v$. 

Let~${G{/}{\equiv_\mathcal{B}}}$ denote the digraph, or multigraph respectively, 
which is obtained from~$G$ by identifying the vertices in the same equivalence class of~$\equiv_\mathcal{B}$ and deleting loops. 
Furthermore, let~${\hat{X} := \{ [x]_{\equiv_\mathcal{B}} \, | \, x \in X \}}$ for every set~${X \subseteq V(D)}$, 
as well as~${\tilde{X} := \{ y \in x \, | \, x \in X \}}$ for every set~${X \subseteq V(G){/}\equiv_\mathfrak{B}}$.

\begin{proposition}
    \label{prop:quotient}
    Let~$G$ be a digraph or a multigraph 
    and let~$\mathcal{B}$ be a set of cuts of~$G$. 
    Then the following statements hold.
    \begin{enumerate}[label=\textit{(\roman*)}]
        \item \label{item:quo-weakly-con} 
            {${G{/}{\equiv_\mathcal{B}}}$ is (weakly) connected if~$G$ is (weakly) connected.}
        \item \label{item:quo-dicut-down} 
            {Every cut (or dicut/bond/dibond, respectively)~${E(X, Y) \in \mathcal{B}}$ of~$G$ 
            is also a cut (or dicut/bond/dibond, respectively) of~${G{/}{\equiv_\mathcal{B}}}$, 
            and~${E(X, Y) = E(\hat{X}, \hat{Y})}$.}
        \item \label{item:quo-dicut-up} 
            {Every cut (or dicut, respectively)~${E(X, Y)}$ of~${G{/}{\equiv_\mathcal{B}}}$ 
            is also a cut (or dicut, respectively) of~$G$, 
            and~${E(X,Y) = E(\tilde{X},\tilde{Y})}$.}
        \item \label{item:quo-dicut-nested} 
            {Two cuts in~$\mathcal{B}$ are nested as cuts of~$G$ if and only if they are nested as cuts of~$G{/}{\equiv_\mathcal{B}}$.}
        \item \label{item:quo-fin-disep} 
            {${G{/}{\equiv_\mathcal{B}}}$ is $\mathcal{B}$-separable.}
    \end{enumerate}
\end{proposition}

\begin{proof}
    For the sake of readability we will phrase the proof just for cuts and bonds. 
    The arguments for dicuts and dibonds are analogous. 
    
    Note that if~${G[X]}$ is (weakly) connected for some~${X \subseteq V(G)}$, then~${G{/}\equiv_\mathcal{B}[\hat{X}]}$ is (weakly) connected as well. 
    Hence statement~\ref{item:quo-weakly-con} is immediate. 

    If~${E(X, Y) \in \mathcal{B}}$, then for every~${x \in X}$ all vertices in~${[x]_{\equiv_\mathcal{B}}}$ are contained in~$X$ by definition of~$\equiv_\mathcal{B}$.
    Analogously, all vertices in~${[y]_{\equiv_\mathcal{B}}}$ lie in~$Y$ for each~${y \in Y}$.
    Hence,~${E(\hat{X}, \hat{Y}) = E(X, Y)}$ and is a cut of~${D{/}{\equiv_\mathcal{B}}}$. 
    If~${E(X,Y)}$ is a bond of~$G$, then so it is as a bond of~$G{/}\equiv_\mathcal{B}$ by the observation on connectivity of the sides from above. 
    This proves statement~\ref{item:quo-dicut-down}. 

    For statement~\ref{item:quo-dicut-up} 
    let~${E(X, Y)}$ be a cut of~${G{/}{\equiv_\mathcal{B}}}$.
    By definition of~$\equiv_\mathcal{B}$ we obtain that~${E(X, Y)}$ is a cut of~$D$ as well as~${M = \tilde{X}}$ and~${N = \tilde{Y}}$ yielding~${E(X, Y) = E(\tilde{X}, \tilde{Y})}$. 
    
    For any subsets~${X, Y \subseteq V(G)}$ if~${X \subseteq Y}$, then~${\hat{X} \subseteq \hat{Y}}$. 
    For any subsets~${X, Y \subseteq V(G){/}{\equiv_\mathcal{B}}}$ if~${X \subseteq Y}$, then~${\tilde{X} \subseteq \tilde{Y}}$.
    With these observations, statement~\ref{item:quo-dicut-nested} is immediate. 

    In order to show statement~\ref{item:quo-fin-disep}, let~$[v]_{\equiv_\mathcal{B}}$ and~$[w]_{\equiv_\mathcal{B}}$ be two different vertices of~${V(G{/}{\equiv_\mathcal{B}})}$. 
    Since~$v$ and~$w$ are not contained in the same equivalence class, there must exist a cut~${E(X, Y)} \in \mathcal{B}$ separating them. 
    By statement~\ref{item:quo-dicut-down} we get that~${E(\tilde{X}, \tilde{Y})}$ is a cut of~$G{/}{\equiv_\mathcal{B}}$ and it separates~$[v]_{\equiv_\mathcal{B}}$ from~$[w]_{\equiv_\mathcal{B}}$ by definition of~$\equiv_\mathcal{B}$.
\end{proof}

We will apply this proposition mostly with the set of all finite bonds of a multigraph~$G$, or the set~$\mathfrak{B}_\textnormal{fin}$ of all finite dibonds of a digraph~$D$, yielding a multigraph which is finitely separable or a digraph which is finitely diseparable.

Let~$D$ be a weakly connected digraph 
and let~$\mathfrak{B}_\textnormal{fin}$ be the set of finite dibonds of~$D$. 
For ease of notation let~$\sim$ denote the relation~${\equiv_{\mathfrak{B}_\textnormal{fin}}}$.

Next we characterise the relation~${v \sim w}$ for any two vertices~${v, w}$. 
An edge set~$W$ is a \emph{witness} for~${v \sim w}$, if it meets every finite cut that separates~$v$ and~$w$ in both directions, i.e.~${W \cap \overrightarrow{E}(X,Y) \neq \emptyset \neq W \cap \overrightarrow{E}(Y,X)}$. 
Hence the existence of a witness for~${v \sim w}$ is an obvious obstruction for the existence of a finite dicut separating~$v$ and~$w$. 
The whole edge set is trivially a witness for~${v \sim w}$. 
Note that there exists always an inclusion-minimal witness for~${v \sim w}$ by Zorn's Lemma.

The following lemmas tell us that given a minimal witness~$W$ for~${v \sim w}$, all vertices incident with an edge of~$W$ are also equivalent to~$v$ with respect to~$\sim$.

\begin{lemma}
    \label{lem:common-witness}
    Let~${v \sim w}$ for two vertices~${v, w \in V(D)}$.
    Then a minimal witness~$W$ for~${v \sim w}$ also witnesses~${v \sim y}$ for any~${y \in V(D[W])}$.
\end{lemma}

\begin{proof}
    Let~$W$ be a minimal witness for~${v \sim w}$.
    Now suppose for a contradiction that there is a~${y \in V(D[W])}$ which is separated from~$v$ by a finite cut~${B = E(X, Y)}$ of~$D$ with~${y \in Y}$ which~$W$ only meets in at most one direction. 
    We consider the case that~${W \cap \overrightarrow{E}(X,Y) = \emptyset}$, the other case is analogous. 
    Since~$W$ witnesses~${v \sim w}$, both vertices~$v$ and~$w$ have to lie on the same side of~$B$, namely~$X$. 
    We claim that~${W' := W \cap E(D[X])}$ also witnesses~${v \sim w}$. 
    This would be a contradiction to the minimality of~$W$ as~$y$ is incident with an edge of~${W \setminus W'}$. 
    
    Let~${E(M, N)}$ be a finite cut of~$D$ separating~$v$ and~$w$, say with~${v \in M}$ and~${w \in N}$. 
    Since~${E(X \cap M, Y \cup N)}$ is a finite cut separating~$v$ and~$w$ but~${W \cap \overrightarrow{E}(X \cap M, Y) = \emptyset}$, we obtain~${W' \cap \overrightarrow{E}(M, N) \neq \emptyset}$. 
    Similarly~${E(X \cap N, Y \cup M)}$ is a finite cut separating~$v$ and~$w$ but~${W \cap \overrightarrow{E}(X \cap N, Y) = \emptyset}$, and hence~${W' \cap \overrightarrow{E}(N, M) \neq \emptyset}$. 
    Thus~$W'$ meets~$E(M,N)$ in both directions, as desired. 
\end{proof}

\begin{corollary}
    \label{cor:min-witness-strongly-connected}
    Let~$D$ be a weakly connected and finitely separable digraph and let~$v$ and~$w$ be two distinct vertices of~$D$ such that~${v \sim w}$. 
    If~$W$ is a finite minimal witness for~${v \sim w}$, then~${D[W]}$ induces a strongly connected digraph. 
\end{corollary}

\begin{proof}
    Since~$v$ and~$w$ are distinct and~$D$ is finitely separable, we know that~$W$ is not the empty set. 
    Assume for a contradiction that there is a dicut~${\overrightarrow{E}(X_W,Y_W)}$ of~${D[W]}$ separating some vertices~${w_1, w_2 \in W}$. 
    Since~$W$ is finite and~$D$ is finitely separable, there exists a finite cut~${E(X,Y)}$ of~$D$ such that~${X_W \subseteq X}$ and~${Y_W \subseteq Y}$.
    By Lemma~\ref{lem:common-witness}, $W$ is also a witness for~${w_1 \sim w_2}$. 
    Hence,~${\overrightarrow{E}(Y,X) \cap W \neq \emptyset}$, contradicting that~${\overrightarrow{E}(X_W,Y_W)}$ is a dicut of~${D[W]}$. 
\end{proof}

We close this subsection with the following corollary of Proposition~\ref{prop:quotient} and Lemma~\ref{lem:2-block}\ref{item:2-block-countable}.

\begin{corollary}
    \label{cor:countable_fin-disep_blocks}
    If~$\mathcal{B}$ is a set of finite cuts of~$G$, then each $2$-block of~${G{/}{\equiv_\mathcal{B}}}$ is countable. 
    \qed
\end{corollary}

\subsection{Quotients of rayless digraphs}

Throughout this subsection let~$D$ be a weakly connected digraph, let~$\mathfrak{B}_\textnormal{fin}$ be the set of finite dibonds of~$D$, and
let~$\mathfrak{B}^{\ast}_\textnormal{fin}$ be the set of finite bonds of~$D$. 
As in the previous subsection, we denote for the sake of readability the relation~$\equiv_{\mathfrak{B}_\textnormal{fin}}$ by~$\sim$.
Moreover, we denote the relation~$\equiv_{\mathfrak{B}^{\ast}_\textnormal{fin}}$ by~$\approx$. 

Note that since~${v \sim w}$ implies that~${v \approx w}$ for all~${v, w \in V(D)}$, we obtain that~$\sim$ induces an equivalence relation on~${V(D{/}{\approx})}$. 
Since moreover the set of finite dibonds of~$D{/}{\approx}$ equals the set of finite dibonds of~$D$ by Proposition~\ref{prop:quotient}, we obtain the following remark.

\begin{remark}
   \label{rem:rayless-double-quotient}
    ${(D{/}{\approx}){/}{\sim} = D{/}{\sim}}$
    \qed
\end{remark}

The aim of this subsection is to show that if~$D$ is rayless, then so is~${D{/}{\sim}}$. 
The analogous statement for the relation~$\approx$ is proven by an easy construction. 

\begin{remark}
   \label{rem:rayless-quotient-finsep}
    If~$D$ is rayless, then~${D{/}{\approx}}$ is rayless as well.
\end{remark}

\begin{proof}
    Suppose for a contradiction that~$D$ is rayless but~${R = [v_0]_\approx [v_1]_\approx \ldots}$ is a ray in~${D{/}{\approx}}$. 
    For each~${i \in \mathbb{N}}$ let~${v'_i \in [v_i]_\approx}$ and~${v''_{i+1} \in [v_{i+1}]_\approx}$ be the endvertices of the edge~${[v_{i}]_\approx[v_{i+1}]_\approx}$ of~$R$ seen in~$D$. 
    To arrive at a contradiction, we will construct a ray in~$D$ inductively. 
    Let~$P_0$ be the trivial path containing just~$v'_0$. 
    Assume for~${i > 0}$ that there is a~${j \geq i}$ such that~$P_i$ is a ${[v_0]_\approx}$--${[v_j]_\approx}$-path which contains~$P_{i-1}$ and is internally disjoint to~$[v_k]_\approx$ for all~${k \geq j}$. 
    Let~$v'''_j$ be the endvertex of~$P_i$ in~$[v_j]_\approx$. 
    By the definition of~$\approx$ and Menger's Theorem there exist infinitely many edge-disjoint $v'''_j$--$v'_j$ paths if~${v'''_j \neq v'_j}$.
    Every~${[v_k]_\approx}$ for~${k \neq j}$ intersects only finitely many of these paths as otherwise~${v_j \approx v_k}$.
    Hence, we can find a $v'''_j$--$v'_j$ path~$P$ which is disjoint from~$P_i$, unless~${v'''_j = v'_j}$ where we set~$P$ to be the trivial path. 
    If~$P$ is disjoint from~${[v_k]_\approx}$ for all~${k > j}$, then let~$P_{i+1}$ be the concatenation of the paths~$P_i$,~$P$ and the edge~${v'_{j} v''_{j+1}}$. 
    Otherwise let~$w$ be the first vertex of~$P$ in~$[v_k]_\approx$ for some~${k > j}$ and let~$P_{i+1}$ be the concatenation of~$P_i$ with~${v'''_j P w}$. 
    In both cases~$P_{i+1}$ satisfies the desired properties and~${\bigcup_{i \in \mathbb{N}} P_i}$ is the desired ray in~$D$.
\end{proof}

Before we can prove the analogue regarding the relation~$\sim$, we have to prepare some lemmas. 
The first is about inclusion-minimal edge sets witnessing the equivalence of two vertices with respect to~$\sim$ in digraphs whose underlying multigraph is rayless.

\begin{lemma}
    \label{lem:rayless-finite}
    Let~$v$ and~$w$ be two distinct vertices of~$D$ with~${v \sim w}$. 
    If~$D$ is rayless and finitely separable,  
    then any minimal edge set of~$D$ witnessing~${v \sim w}$ is finite and non-empty. 
\end{lemma}

\begin{proof}
    Let~${W \subseteq E(D)}$ be an inclusion-minimal witness for~${v \sim w}$. 
    Note that since~$D$ is finitely separable,~$W$ is non-empty. 
    Let us consider the $2$-block-cutvertex tree~$T$ of~$D$. 
    Let~$P$ denote the finite path in~$T$ whose endvertices are the $2$-blocks of~$D$ containing~$v$ and~$w$, respectively. 
    By Remark~\ref{rem:bonds-in-blocks}, each bond of~$D$ separating~$v$ and~$w$ is a bond of the finitely many $2$-blocks corresponding to the vertices of~$P$.
    This implies that all edges in~$W$ are contained in the finitely many $2$-blocks which correspond to vertices of~$P$.
    However, each $2$-block of~$D$ is finite since~$D$ is finitely separable and rayless 
    and such multigraphs do not have infinite $2$-blocks by Lemma~\ref{lem:2-block}\ref{item:2-block-rayless}. 
    So~$W$ is contained in a finite set and thus finite itself. 
\end{proof}

\begin{proposition}
    \label{prop:rayless-quotient}
    If~$D$ is rayless, then so is~${D{/}{\sim}}$.
\end{proposition}

\begin{proof}
    By Remarks~\ref{rem:rayless-double-quotient} and~\ref{rem:rayless-quotient-finsep} we may assume without loss of generality that~$D$ is finitely separable. 
    Suppose for a contradiction that~$D$ is rayless but~${R = [v_0]_\sim[v_1]_\sim \ldots}$ is a ray in~${D{/}{\sim}}$. 
    For each~${i \in \mathbb{N}}$ let~${v'_i \in [v_i]_\sim}$ and~${v''_{i+1} \in [v_{i+1}]_\sim}$ be the endvertices of the edge~${[v_{i}]_\sim [v_{i+1}]_\sim \in E(R)}$ seen in~$D$.
    Furthermore, let~$W_i$ be an inclusion-minimal witness for~${v''_{i} \sim v'_{i+1}}$ for every~${i \in \mathbb{N}}$ with~${i \geq 1}$.
    We know by Lemma~\ref{lem:common-witness} that each~$W_i$ is completely contained in~$[v_i]_\sim$.
    By Lemma~\ref{lem:rayless-finite} and Corollary~\ref{cor:min-witness-strongly-connected} each~$W_i$ is finite, non-empty, and strongly connected. 
    Since each~$W_i$ is completely contained in~${[v_i]_\sim}$, we get that 
   ~${W_i \cap W_j = \emptyset}$ holds for all~${i, j \in \mathbb{N}}$ with~${i \neq j}$.
    Let~${P_i}$ be a directed $v''_{i}$--$v'_{i+1}$~path that is contained in~$W_i$ for every~${i \in \mathbb{N}}$ with~${i \geq 1}$. 
    Now the union of these paths together with the edges between~$v'_i$ and~$v''_{i+1}$ 
    is a ray in~$D$, a contradiction. 
\end{proof}

\section{Comparing Conjecture~\ref{conj:LY-inf} with Conjecture~\ref{conj:LY-inf-nested}}
\label{sec:examples}

In this section we shall compare Conjecture~\ref{conj:LY-inf} with  Conjecture~\ref{conj:LY-inf-nested} more closely by looking at two examples. 
In both examples we will see an indication why Conjecture~\ref{conj:LY-inf-nested} might be properly stronger than Conjecture~\ref{conj:LY-inf}. 
To put it straight, both examples show the following: 

\vspace{5pt}

\begin{center}
    \emph{There exist finitary dijoins that are part of an optimal pair, but of no nested optimal pair.}
\end{center}

\vspace{5pt}

\noindent 
This is severely different from finite digraphs. 
There, we could always keep the dijoin~$F$ of any optimal pair~${(F, \mathcal{B})}$ and just iteratively `uncross' all dicuts of~$\mathcal{B}$,  yielding a set~$\mathcal{B}'$ of nested disjoint dicuts such that~${(F, \mathcal{B}')}$ is a nested optimal pair. 
We illustrate this uncrossing process in the proof of Lemma~\ref{lem:fin_parameter_equi}. 

Let us now describe the first example. 

\begin{example}
    \label{ex:zackenstrahl}
    Consider the infinite weakly connected digraph~$D_1$ depicted twice in Figure~\ref{fig:zackenstrahl}. 
    Before we analyse~$D_1$ in detail, let us define~$D_1$ properly.
    
    Let~${A = \{ a_i \, | \, i \in \mathbb{N} \}}$ and~${B = \{ b_i \, | \, i \in \mathbb{N} \}}$ be two disjoint countably infinite sets.
    Additionally, let~$r$ be some set which is neither contained in~$A$ nor in~$B$.
    Now we set
    \[
        {V(D_1) := A \cup B \cup \{ r \} }.
    \]
    Next we define the 
    sets~${E_1 := \{ a_i b_i  \, | \, i \in \mathbb{N} \}}$, 
    ${E_2 := \{ a_i b_{i+1}  \, | \, i \in \mathbb{N} \}}$ 
    and~${E_3 := \{ b_i r \, | \, i \in \mathbb{N} \}}$.
    We complete the definition of~$D_1$ by setting 
    \[
        {E(D_1) := E_1 \cup E_2 \cup E_3}.
    \]
    
    \begin{figure}[htbp]
        \centering
        \includegraphics{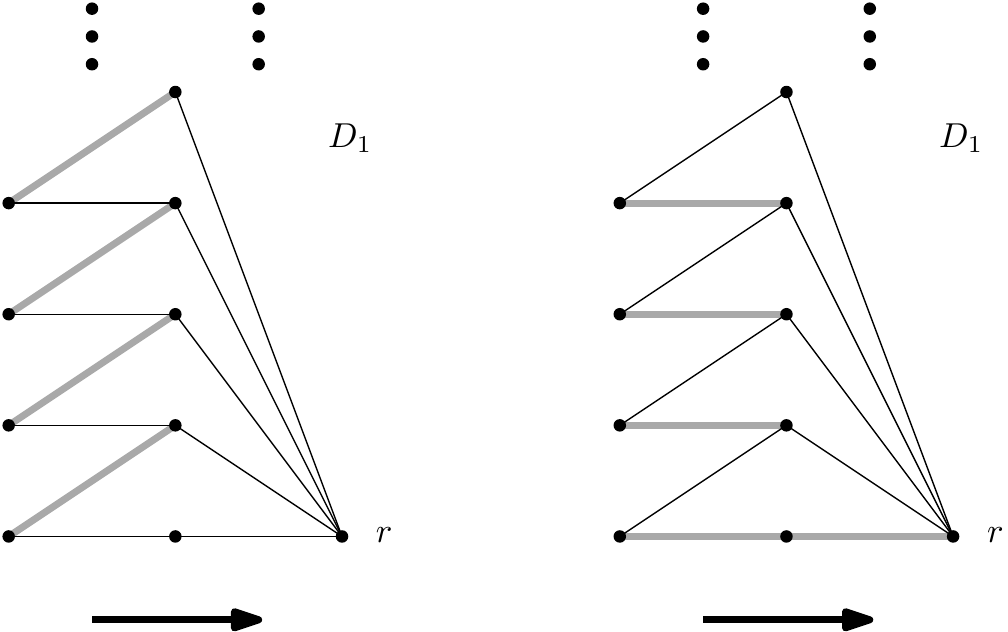}
        \caption{
            Two instances of the digraph~$D_1$. 
            All edges are meant to be directed from left to right. 
            The grey edges in the left instance of~$D_1$ form a finitary dijoin featuring in a nested optimal pair for~$D_1$. 
            The grey edges in the right instance form a finitary dijoin featuring in an optimal pair for~$D_1$, but not in any nested optimal pair for~$D_1$. }
        \label{fig:zackenstrahl}
    \end{figure}
    
    Next consider the set~$E_2$ of grey edges in the left instance of~$D_1$ depicted in Figure~\ref{fig:zackenstrahl}, call it~$F_L$. 
    It is easy to check that~$F_L$ forms a finitary dijoin of~$D_1$. 
    Furthermore, we can easily find a nested optimal pair for~$D_1$ in which~$F_L$ features.
    Hence,~$D_1$ is not a counterexample to Conjecture~\ref{conj:LY-inf-nested}.
    
    In the right instance of~$D_1$ depicted in Figure~\ref{fig:zackenstrahl}, the set of grey edges~${E_1 \cup \{ b_0 r \}}$, call it~$F_R$, also forms a finitary dijoin. 
    And again we can easily find an optimal pair for~$D_1$ in which~$F_R$ features.
    However, no matter which finite dicut we choose which contains the grey edge adjacent to~$r$, it cannot be nested with all the finite dicuts we choose for all the other edges of~$F_R$.
    Therefore,~$F_R$ does not feature in any nested optimal pair for~$D_1$.
\end{example}

Let us now consider another example, witnessing the same behaviour of finitary dijoins as Example~\ref{ex:zackenstrahl} does. 
However, the structure of the digraph~$D_2$ in the following example is rather different from~$D_1$. 
In particular,~$D_2$ is a locally finite digraph, i.e.~every vertex is incident with only finitely many edges.

\begin{example}
    \label{ex:stufengitter}
    Consider the infinite weakly connected digraph~$D_2$ depicted in Figure~\ref{fig:stufengitter}. 
    We first define vertex set of~$D_2$ as 
    \[
        {V(D_2) := \left \{ (x,y) \in \mathbb{Z} \times \mathbb{Z} \, \middle| \, \frac{x}{2} - y \leq 1 \right \}}.
    \]
    Note that for each~${(x,y) \in V(D_2)}$ both~${(x, y+1)}$ and~${(x-1, y)}$ are in~${V(D_2)}$ as well. 
    We define 
    ${E_1 := \{ (x, y+1) (x,y) \, | \, (x,y) \in V(D_2) \}}$ and 
    ${E_2 := \{ (x-1,y) (x,y) \, | \, (x,y) \in V(D_2) \}}$. 
    Finally, we define the edge set of~$D_2$ by 
    \[
        {E(D_2) := E_1 \cup E_2}.
    \]

    \begin{figure}[htbp]
        \centering
        \includegraphics{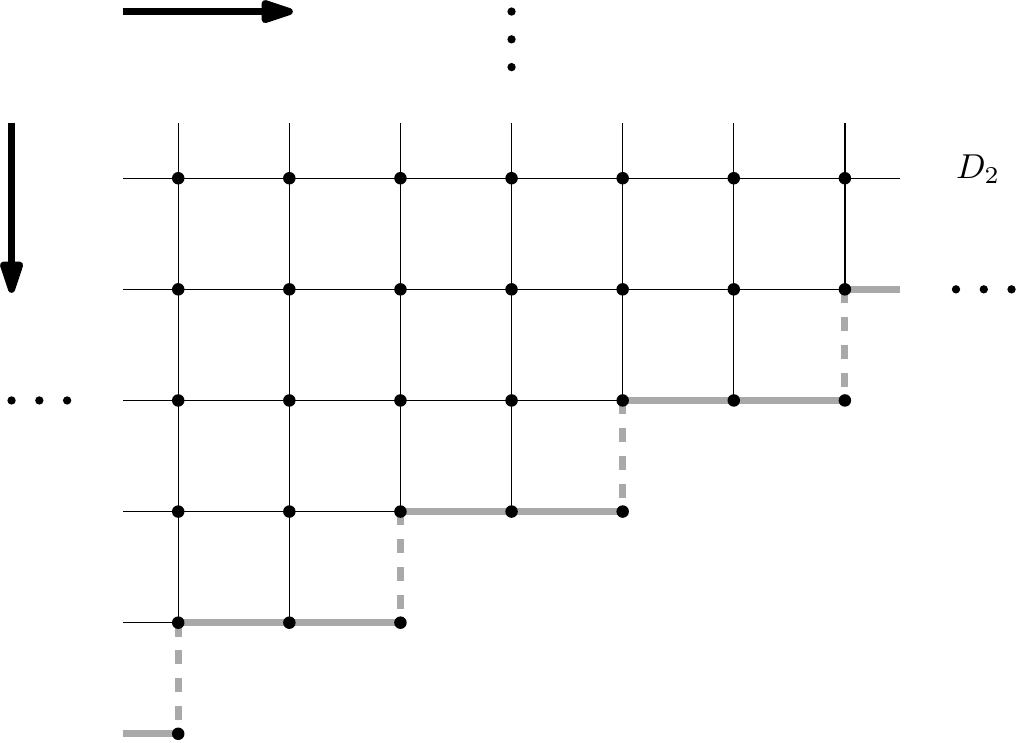}
        \caption{
            The digraph~$D_2$. 
            The edges are meant to be directed from left to right and from top to bottom. 
            The dashed grey edges form a finitary dijoin featuring in a nested optimal pair for~$D_2$. 
            The solid grey edges form a finitary dijoin featuring in an optimal pair for~$D_2$, but not in any nested optimal pair for~$D_2$.}
        \label{fig:stufengitter}
    \end{figure}
    
    Now consider the set of dashed grey edges in Figure~\ref{fig:stufengitter}, \[
        {F_d := \left\{  (x, y+1) (x,y) \, \middle| \, \frac{x}{2} - y = 1 \right\}}.
    \]
    It is an easy exercise to check that~$F_d$ forms a finitary dijoin of~$D_2$ which also features in a nested optimal pair for~$D_2$. 
    Therefore, the digraph~$D_2$ is also no counterexample to Conjecture~\ref{conj:LY-inf-nested}. 
    
    In contrast to this, let us now consider the set of uninterruptedly grey edges in Figure~\ref{fig:stufengitter}, 
    \[
        {F_s := \left\{ (x-1,y) (x,y) \, \middle| \, \frac{x}{2} - y = \frac{1}{2} \, \textnormal{ or } \, \frac{x}{2} - y = 1 \right\}}.
    \]
    Again it is easy to check that~$F_s$ forms a finitary dijoin of~$D_2$. 
    However,~$F_s$ is not part of any nested optimal pair for~$D_2$. 
    This is not difficult to prove using the fact that~$F_d$ is a finitary dijoin of~$D_2$ as well. 
    We leave this proof to the reader.
\end{example}

\section{Reductions for the Infinite Lucchesi-Younger Conjecture}
\label{sec:reduce}

In this section we prove some reductions for Conjecture~\ref{conj:LY-inf} and Conjecture~\ref{conj:LY-inf-nested} in the sense that it suffices to solve these conjectures on a smaller class of digraphs. 
We begin by reducing these conjectures to finitely diseparable digraphs via the following lemma.

\begin{lemma}
    \label{lem:fin-disep}
    Let~$D$ be a weakly connected digraph 
    and~$\mathfrak{B}$ be a 
    class of dibonds of~$D$.
    Then~${(F, \mathcal{B})}$ is a (nested) $\mathfrak{B}$-optimal pair for~$D$ if and only if it is a (nested) $\mathfrak{B}$-optimal pair for~${D{/}\equiv_\mathfrak{B}}$.
\end{lemma}

\begin{proof}
    Note first that by Proposition~\ref{prop:quotient}, $D{/}{\equiv_\mathfrak{B}}$ is weakly connected and that~$\mathfrak{B}$ is also a set of dibonds of~${D{/}{\equiv_\mathfrak{B}}}$. 
    
    Suppose~${(F, \mathcal{B})}$ is a (nested) $\mathfrak{B}$-optimal pair for~$D$. 
    Then~$F$ is still a subset of~${E(D{/}{\equiv_\mathfrak{B}})}$ since each edge of~$F$ lies on some dibond~${B \in \mathcal{B} \subseteq \mathfrak{B}}$. 
    Hence,~$F$ is still a $\mathfrak{B}$-dijoin of~${D{/}{\equiv_\mathfrak{B}}}$, and~${(F, \mathcal{B})}$ is indeed a (nested) $\mathfrak{B}$-optimal pair for~${D{/}{\equiv_\mathfrak{B}}}$, again by Proposition~\ref{prop:quotient}. 
    
    Similarly, if~${(F, \mathcal{B})}$ is a (nested) $\mathfrak{B}$-optimal pair for~${D{/}\equiv_\mathfrak{B}}$, then so it is for~$D$, again by Proposition~\ref{prop:quotient}.
\end{proof}

The next reduction of Conjecture~\ref{conj:LY-inf} and Conjecture~\ref{conj:LY-inf-nested} tells us that we can restrict our attention also to digraphs whose underlying multigraph is $2$-connected.

\begin{lemma}
    \label{lem:2-con-L-Y}
    Let~$D$ be a weakly connected digraph and~$\mathfrak{B}$ be a 
    class of dibonds of~$D$. 
    Let~$\mathcal{X}$ denote the set of all $2$-blocks of~$D$. 
    Then the following statements are true. 
    \begin{enumerate}[label=\textit{(\roman*)}]
        \item \label{item:2-conn-0}
            {For each~${X \in \mathcal{X}}$ the set~${\mathfrak{B}_X := \{ B \in \mathfrak{B} \, | \, B \subseteq E(X) \}}$ is a 
            class of dibonds of~$X$ 
            and~${\mathfrak{B} = \dot{\bigcup}_{X \in \mathcal{X}} \mathfrak{B}_X}$. 
            Moreover, if~$\mathfrak{B}$ is finitely corner-closed, then so is~${\mathfrak{B}_X}$.}
        \item \label{item:2-conn-1} 
            {If~${(F, \mathcal{B})}$ is a (nested) $\mathfrak{B}$-optimal pair for~$D$, 
            then~${(F_X, \mathcal{B}_X)}$ is a (nested) $\mathfrak{B}_X$-optimal pair for every~${X \in \mathcal{X}}$, 
            where~${F_X := F \cap E(X)}$ and~${\mathcal{B}_X := \{ B \in \mathcal{B} \, | \, B \subseteq E(X) \}}$.}
        \item \label{item:2-conn-2} 
            {If~${(F_X, \mathcal{B}_X)}$ is a (nested) $\mathfrak{B}_X$-optimal pair for every~${X \in \mathcal{X}}$, 
            then~${(F, \mathcal{B})}$ is a (nested) $\mathfrak{B}$-optimal pair for~$D$, 
            where~${F := \bigcup_{X \in \mathcal{X}} F_X}$ and~${\mathcal{B} := \bigcup_{X \in \mathcal{X}} \mathcal{B}_X}$.}
    \end{enumerate}
\end{lemma}

\begin{proof}
    Let~$X$ be a $2$-block of~$D$. 
    By Remark~\ref{rem:bonds-in-blocks} every dibond~${B \in \mathfrak{B}}$ is either contained in~$E(X)$ and hence a dibond of~$X$, or disjoint to~$E(X)$. 
    Vice versa, every dibond of~$X$ is a dibond of~$D$ as well. 
    Statement~\ref{item:2-conn-0} is now easy to check. 
    
    For statement~\ref{item:2-conn-1}, 
    let~${X \in \mathcal{X}}$ and let~${(F, \mathcal{B})}$ be a (nested) $\mathfrak{B}$-optimal pair for~$D$. 
    Then by just translating the definitions we obtain that~${(F \cap E(X), \mathcal{B}_X)}$ is a (nested) $\mathfrak{B}_X$-optimal pair for~$D$, as well as for~$X$. 
    
    Now we show that statement~\ref{item:2-conn-2} is true.
    So let us assume that~${(F_X, \mathcal{B}_X)}$ is a (nested) $\mathfrak{B}_X$-optimal pair for every~${X \in \mathcal{X}}$. 
    With statement~\ref{item:2-conn-0} (and Remark~\ref{rem:bonds-in-blocks}\ref{item:bonds-in-blocks2}) we immediately get that with~${(F, \mathcal{B})}$ is a (nested) $\mathfrak{B}$-optimal pair for~$D$. 
\end{proof}

We can now close this section by proving Theorem~\ref{thm:reduction}.
In order to do this we basically only need to combine Lemma~\ref{lem:fin-disep} and Lemma~\ref{lem:2-con-L-Y}.
Let us restate the theorem.

\setcounter{section}{1}
\setcounter{thm}{5}

\begin{theorem}
    If Conjecture~\ref{conj:LY-inf} (or Conjecture~\ref{conj:LY-inf-nested}, respectively) holds for all countable finitely diseparable digraphs whose underlying multigraph is $2$-connected, then Conjecture~\ref{conj:LY-inf} (or Conjecture~\ref{conj:LY-inf-nested}, respectively) holds for all weakly connected digraphs. 
\end{theorem}

\setcounter{section}{4}
\setcounter{thm}{2}

\begin{proof}
    Let~$D$ be any weakly connected digraph 
    and let~${\mathfrak{B}_\textnormal{fin}}$ the set of finite dibonds of~$D$.
    We know by Proposition~\ref{prop:quotient} that~${D{/}{\equiv_{\mathfrak{B}_\textnormal{fin}}}}$ is a weakly connected and finitely diseparable digraph, and so is every $2$-block of it.
    Furthermore, Corollary~\ref{cor:countable_fin-disep_blocks} yields that each $2$-block of~${D{/}{\equiv_{\mathfrak{B}_\textnormal{fin}}}}$ is countable.
    By our assumption we know that Conjecture~\ref{conj:LY-inf} (or Conjecture~\ref{conj:LY-inf-nested}, respectively) holds for every countable $2$-block of~${D{/}{\equiv_{\mathfrak{B}_\textnormal{fin}}}}$.
    So using Lemma~\ref{lem:2-con-L-Y} we obtain a (nested) optimal pair for~${D{/}{\equiv_{\mathfrak{B}_\textnormal{fin}}}}$.
    Then we also obtain an optimal pair for~$D$ by Lemma~\ref{lem:fin-disep}.
\end{proof}

\section{Special cases}
\label{sec:cases}

In this section we prove some special cases of Conjecture~\ref{conj:LY-inf-nested}, or more precisely cases of Question~\ref{q:general-LY}. 

\subsection{Finite parameters}

Let~$D$ be a 
weakly connected 
digraph. 
Let~$\mathfrak{B}$ be a 
class of finite dibonds of~$D$.
Before we come to the first special case, we state a basic observation.

\begin{lemma}
    \label{lem:fin_parameter_equi}
    The following statements are equivalent:
    \begin{enumerate}[label=\textit{(\roman*)}]
        \item \label{item:fpe-1} 
            {$D$ admits a $\mathfrak{B}$-dijoin of finite size.}
        \item \label{item:fpe-2} 
            {The maximal number of disjoint dibonds in~$\mathfrak{B}$ is finite.}
    \end{enumerate}
    If~$\mathfrak{B}$ is finitely corner-closed, then~\ref{item:fpe-1} and~\ref{item:fpe-2} are also equivalent with the following statement:
    \begin{enumerate}[label=\textit{(\roman*)}, resume]
        \item \label{item:fpe-3} 
            {The maximal number of disjoint and pairwise nested dibonds in~$\mathfrak{B}$ is finite.}
    \end{enumerate}
\end{lemma}

\begin{proof}
    We start by proving the implication from~\ref{item:fpe-1} to~\ref{item:fpe-2}. 
    Let~$F$ be a $\mathfrak{B}$-dijoin of~$D$ of finite size. 
    Then, by definition, we can find at most~$\abs{F}$ many disjoint dibonds in~$\mathfrak{B}$. 

    For the implication~\ref{item:fpe-2} to~\ref{item:fpe-1} note that for any inclusion-wise maximal set~$\mathcal{B}$ of disjoint dibonds in~$\mathfrak{B}$ the set~${F := \bigcup \mathcal{B}}$ is a finite $\mathfrak{B}$-dijoin of~$D$. 
    
    The implication from~\ref{item:fpe-2} to~\ref{item:fpe-3} is immediate, 
    even if~$\mathfrak{B}$ is not finitely corner-closed. 
    
    Finally, we assume statement~\ref{item:fpe-3} and that~$\mathfrak{B}$ is finitely corner-closed, and we prove statement~\ref{item:fpe-1}. 
    
    Suppose that for some finite set~$\mathcal{B} \subseteq \mathfrak{B}$ of pairwise disjoint and pairwise nested finite dibonds which is of maximum size there is some dibond~${A \in \mathfrak{B}}$ which is disjoint to each dibond in~$\mathcal{B}$. 
    Without loss of generality, let~$\mathcal{B}$ and~$A$ be chosen such that the set of dibonds in~$\mathcal{B}$ that cross~$A$ is of minimum size among all possible choices. 
    
    Let~${B \in \mathcal{B}}$ be chosen such that~$A$ and~$B$ cross and either~$\insh(B)$ (first case) or~$\outsh(B)$ (second case) is inclusion-minimal among all sides of the elements of~$\mathcal{B}$ that cross~$A$.
    
    In the first case we consider the dicut~${A \wedge B \in \mathfrak{B}^\oplus}$. 
    Note that since both~$A$ and~$B$ are dibonds, the out-shore of~${A \wedge B}$ induces a weakly connected digraph. 
    Hence an easy case analysis shows that any dibond in its decomposition into dibonds in~$\mathfrak{B}$ is nested with every dibond in~$\mathcal{B}$ as well as with each other. 
    In particular,~${A \wedge B}$ is a dibond in~$\mathfrak{B}$, since otherwise it would contradict the maximality of~$\mathcal{B}$. 
    Moreover, let~${A'}$ be any dibond appearing in the decomposition of~${A \vee B}$ into dibonds in~$\mathfrak{B}$. 
    As before, we can show that~$A'$ is nested with~${A \wedge B}$, as well as with any dibond in~${\mathcal{B}}$ which is nested with~$A$. 
    And since~${\mathcal{B}' := (\mathcal{B} \setminus \{ B \}) \cup \{ A \wedge B \}}$ is a set of pairwise disjoint dibonds in~$\mathfrak{B}$ and~$A'$ crosses strictly fewer dicuts in~$\mathcal{B}'$ than~$A$ crosses in~$\mathcal{B}$, the pair~$\mathcal{B}'$ and~$A'$ contradicts the choice of~$\mathcal{B}$ and~$A$. 
    In the second case the same argument works with the roles of~${A \wedge B}$ and~${A \vee B}$ reversed. 
    
    In any case, this contradicts the existence of such a set~$\mathcal{B}$ and such a dibond~$A$. 
    Therefore, for any finite set~$\mathcal{B} \subseteq \mathfrak{B}$ of pairwise disjoint and pairwise nested finite dibonds which is of maximum size the set~${\bigcup \mathcal{B}}$ is a finite $\mathfrak{B}$-dijoin.
\end{proof}

\vspace{0.2cm}

Given an edge set~${N \subseteq E(D)}$, let~${\mathfrak{B}\restricted{N}}$ denote the set~${\{ B \in \mathfrak{B} \, | \, B \subseteq N \}}$. 
Note that~$\mathfrak{B}\restricted{N}$ is a class of finite dibonds of the contraction minor~${D.N}$ and if~$\mathfrak{B}$ is finitely corner-closed, then so is~$\mathfrak{B}\restricted{N}$. 
The following lemma uses a standard compactness argument to show the existence of a (nested) optimal pair for~$D$ based on the existence of (nested) optimal pairs of bounded size for its finite contraction minors.

\begin{lemma}
    \label{lem:fin-parameter-compactness}
    Let~${n \in \mathbb{N}}$. 
    If for every finite~${N \subseteq E(D)}$ there is a (nested) ${\mathfrak{B}\restricted{N}}$-optimal pair~${(F_N, \mathcal{B}_N)}$ for~${D.N}$ with~${\abs{F_N} \leq n}$, then there is a (nested) $\mathfrak{B}$-optimal pair~${(F, \mathcal{B})}$ for~$D$. 
\end{lemma}

\begin{proof}
    Let~$\mathcal{B}$ be a maximal (nested) set of disjoint dicuts in~$\mathfrak{B}$. 
    Note that~${\abs{\mathcal{B}} \leq n}$, since otherwise a subset~${\mathcal{B}' \subseteq \mathcal{B}}$ of size~${n+1}$ would contradict the assumption for~${N = \bigcup \mathcal{B}'}$. 
    
    Let~${N \subseteq E(D)}$ be a finite set of edges such that~${\bigcup \mathcal{B} \subseteq N}$ holds. 
    Since~${D.N}$ is a finite weakly connected digraph, 
    there exists a (nested) ${\mathfrak{B}\restricted{N}}$-optimal pair~${(F_N, \mathcal{B}_N)}$ for~${D.N}$ by assumption. 
    By the choice of~$N$ and Lemma~\ref{lem:contr-cuts} we know that each element of~$\mathcal{B}$ is also a finite dicut of~${D.N}$. 
    Furthermore, each finite dicut in~${D.N}$ is also one in~$D$ and, thus,~${\mathcal{B}_N}$ is a set of disjoint finite dicuts in~$D$. 
    Hence,~${\abs{\mathcal{B}} = \abs{\mathcal{B}_N} = \abs{F_N}}$. 
    Using that the elements in~$\mathcal{B}$ are pairwise disjoint (and nested) finite dicuts, we get that~${(F_N, \mathcal{B})}$ is a (nested) ${\mathfrak{B}\restricted{N}}$-optimal pair for~${D.N}$ as well.
    Given a finite edge set ${M \supseteq N}$ with a (nested) ${\mathfrak{B}\restricted{M}}$-optimal pair~${(F_M, \mathcal{B}_M)}$ for~${D.M}$ we obtain that~${(F_M, \mathcal{B})}$ is also a nested optimal pair for~${D.N}$. 
    
    Note that for any finite edge set~${N \subseteq E(D)}$ satisfying~${\bigcup \mathcal{B} \subseteq N}$ there are only finitely many possible edge sets~${F_N \subseteq \bigcup \mathcal{B}}$ such that~${(F_N, \mathcal{B})}$ is a (nested) ${\mathfrak{B}\restricted{N}}$-optimal pair for~${D.N}$.
    Hence, we get via the compactness principle an edge set~${F \subseteq \bigcup \mathcal{B}}$ with~${\abs{F \cap B} = 1}$ for every~${B \in \mathcal{B}}$ such that~${(F, \mathcal{B})}$ is a (nested) ${\mathfrak{B}\restricted{M}}$-optimal pair for~${D.M}$ for every finite edge set~${M \subseteq E(D)}$ satisfying~${\bigcup \mathcal{B} \subseteq M}$. 
    
    We claim that~${(F, \mathcal{B})}$ is a (nested) $\mathfrak{B}$-optimal pair for~$D$.
    We already know by definition that~$\mathcal{B}$ is a (nested) set of disjoint finite dicuts in~$\mathfrak{B}$ 
    and that~${F \subseteq \bigcup \mathcal{B}}$ with~${\abs{F \cap B} = 1}$ for every~${B \in \mathcal{B}}$. 
    It remains to check that~$F$ is a $\mathfrak{B}$-dijoin of~$D$. 
    So let~${B' \in \mathfrak{B}}$. 
    Then the set~${N' := B' \cup \bigcup \mathcal{B}}$ is also finite and~$B'$ is a finite dicut of~${D.N'}$.
    Since~${(F, \mathcal{B})}$ is also a nested optimal pair for~${D.N'}$, we know that~${F \cap B' \neq \emptyset}$ holds, which proves that~$F$ is a $\mathfrak{B}$-dijoin of~$D$. 
\end{proof}

Hence, Lemmas~\ref{lem:fin_parameter_equi} and~\ref{lem:fin-parameter-compactness} together with Theorem~\ref{thm:LY-fin-nested} yield Theorem~\ref{thm:LY-inf-cases}~\ref{item:LY-fin-dijoin},~\ref{item:LY-fin-disj-dicuts} and~\ref{item:LY-fin-nested-disj-dicuts}.

\subsection{Every edge lies in only finitely many dibonds and reductions to this case}

We continue with verifying another special case of Question~\ref{q:general-LY}.
The proof is also based on a compactness argument.
However, we need to choose the set up for the argument more carefully.

\begin{lemma}
    \label{L-Y-fin_many_fin_dibonds}
    Conjecture~\ref{conj:LY-inf-nested} holds for weakly connected digraphs in which every edge lies in only finitely many finite dibonds.
\end{lemma}

\begin{proof}
    Let~$D$ be a weakly connected digraph where every edge lies in only finitely many finite dibonds. 
    For an edge~${e \in E(D)}$ let~$\mathcal{B}_e$ denote the set of finite dibonds of~$D$ that contain~$e$. 
    Our assumption on~$D$ implies that~$\mathcal{B}_e$ is a finite set. 
    For a finite set~$\mathcal{B}$ of finite dibonds of~$D$ we define~${\hat{\mathcal{B}} = \bigcup \{ \mathcal{B}_e \, | \, e \in \bigcup \mathcal{B} \}}$. 
    Again our assumption on~$D$ implies that~$\hat{\mathcal{B}}$ is finite. 
    Note that~${\mathcal{B} \subseteq \hat{\mathcal{B}}}$ holds. 
    
    Given a finite set~$\mathcal{B}$ of finite dibonds of~$D$, we call~${(F_{\mathcal{B}}, \mathcal{B}')}$ a \emph{nested pre-optimal pair} 
    for~$\mathcal{B}$ 
    if the following hold: 
    \begin{enumerate}
        \item $F_{\mathcal{B}}$ intersects every element of~$\mathcal{B}$, 
        \item ${\mathcal{B}' \subseteq \hat{\mathcal{B}}}$, 
        \item the elements of~$\mathcal{B}'$ are pairwise nested and disjoint, 
        \item ${F_{\mathcal{B}} \subseteq \bigcup \mathcal{B}}$, and 
        \item ${\abs{F_{\mathcal{B}} \cap B'} = 1}$ for every~${B' \in \mathcal{B}'}$. 
    \end{enumerate}

    Now let~$\mathcal{B}_1$ and~$\mathcal{B}_2$ be two finite sets of finite dibonds of~$D$ with~${\mathcal{B}_1 \subseteq \mathcal{B}_2}$, and let~${(F_{\mathcal{B}_2}, \mathcal{B}_2')}$ be a nested pre-optimal pair  
    for~$\mathcal{B}_2$. 
    Then it is easy to check that 
    \[
        ( F_{\mathcal{B}_2}, \mathcal{B}'_2) \restricted \mathcal{B}_1 :=
        \Big( F_{\mathcal{B}_2} \cap \bigcup \mathcal{B}_1, 
        \big\{ B \in \mathcal{B}_2' \ |\ \abs{B \cap (F_{\mathcal{B}_2} \cap \bigcup \mathcal{B}_1)} = 1  \big\} \Big)
    \]
    is a nested pre-optimal pair 
    for~$\mathcal{B}_1$. 
    
    We know that for every finite set~$\mathcal{B}$ of finite dibonds of~$D$ there exists a nested pre-optimal pair 
    for~$\mathcal{B}$, since for a nested optimal pair~${(F_{\hat{\mathcal{B}}},\mathcal{B}')}$ for~${D.(\bigcup \hat{\mathcal{B}})}$, which exists by Theorem~\ref{thm:LY-fin-nested}, 
    we have that~${(F_\mathcal{B}, \mathcal{B}') \restricted \mathcal{B}}$
    is a nested pre-optimal pair 
    for~$\mathcal{B}$. 
    However, there can only be finitely many nested pre-optimal pairs 
    for~$\mathcal{B}$ 
    as both~${\bigcup \mathcal{B}}$ and~$\hat{\mathcal{B}}$ are finite. 
    
    Now we get by the compactness principle an edge set~${F'_D \subseteq E(D)}$ and a set~$\mathcal{B}_D$ of finite dibonds of~$D$ 
    such that
    for every finite set~$\mathcal{B}$ of finite dibonds of~$D$, we have that~${(F'_D, \mathcal{B}_D) \restricted \mathcal{B}}$
    is a nested pre-optimal pair 
    for~$\mathcal{B}$.
    Furthermore, 
    let~$F_D$ be the subset of~$F'_D$ consisting of all elements of~$F'_D$ that lie on a finite dibond of~$D$. 
    Note that for every finite set~$\mathcal{B}$ of finite dibonds of~$D$ the pair~$(F_D, \mathcal{B}_D) \restricted \mathcal{B}$ is still a nested pre-optimal pair 
    for~$\mathcal{B}$. 
    We claim that~${(F_D, \mathcal{B}_D)}$ is a nested optimal pair for~$D$. 
    
    First we verify that~$F_D$ is a finitary dijoin of~$D$. 
    Let~$B$ be any finite dibond of~$D$. 
    Then~$F_D$ meets~$B$, 
    because~${(F_D, \mathcal{B}_D) \restricted \{B\}}$ 
    is a nested pre-optimal pair for~${\{B\}}$. 
    So~$F_D$ is a finitary dijoin of~$D$. 
    
    Next consider any element~${e \in F_D}$. 
    By definition of~$F_D$ we know that~${e \in B_e}$ holds for some finite dibond~$B_e$ of~$D$. 
    Using again 
    that~${(F_D,\mathcal{B}_D) \restricted \{B_e\}}$
    is a nested pre-optimal pair for~${\{B_e\}}$, we get that~${e \in \bigcup \mathcal{B}_D}$. 
    So the inclusion~${F_D \subseteq \bigcup \mathcal{B}_D}$ is valid. 
    
    Given any~${B_D \in \mathcal{B}_D}$ we know 
    that~${(F_D, \mathcal{B}_D) \restricted \{B_D\}}$ 
    is a nested pre-optimal pair for~${\{B_D\}}$. 
    Hence,~${\abs{F_D \cap B} = 1}$ holds for every~${B \in \mathcal{B}_D}$.     
    Finally, let us consider two arbitrary but different elements~$B_1$ and~$B_2$ of~$\mathcal{B}_D$. 
    We know 
    that~${(F_D, \mathcal{B}_D) \restricted \{B_1, B_2\}}$ 
    is a nested pre-optimal pair for~${\{B_1, B_2\}}$. 
    Therefore,~$B_1$ and~$B_2$ are disjoint and nested. 
    This shows that~${(F_D, \mathcal{B}_D)}$ is a nested optimal pair for~$D$ and completes the proof of this lemma. 
\end{proof}

The next lemma can be used together with Lemma~\ref{L-Y-fin_many_fin_dibonds} to deduce that Conjecture~\ref{conj:LY-inf-nested} holds for weakly connected digraphs without infinite dibonds.

\begin{lemma}
    \label{lem:inf-many-fin-dibonds->inf-dicut}
    In a weakly connected digraph~$D$ where some edge~$e$ lies in infinitely many finite dibonds of~$D$ there is an infinite dibond containing~$e$.
\end{lemma}

\begin{proof}
   We construct with a compactness argument a dibond containing~${e =: vw}$ that is distinct from every finite dibond. 
   
   Let~${W \subseteq V(D)}$ be finite with~${v,w \in W}$. 
   Consider the set~$\mathcal{B}_W$ consisting of those bipartitions~${(A,B)}$ of~$W$ with~${v \in A}$ and~${w \in B}$ such that~${\overrightarrow{E}(B,A)}$ is empty, but~${\overrightarrow{E}(A,B)}$ contains no finite dibond of~$D$. 
   Obviously,~$\mathcal{B}_W$ is finite. 
   For any dibond~${\overrightarrow{E}(X,Y)}$ containing~$e$ that is not contained in~${E(D[W])}$ the bipartition~${(X \cap W, Y \cap W)}$ is in~$\mathcal{B}_W$. 
   And since~$e$ lies in infinitely many dibonds, such a dibond always exists. 
   Moreover, for~${W \subseteq W'}$ and~${(A,B) \in \mathcal{B}_{W'}}$ we have~${(A \cap W, B \cap W) \in \mathcal{B}_{W}}$. 
   Hence by compactness there is a bipartition~$(A,B)$ of~$V(D)$ such that~${(A \cap W, B \cap W) \in \mathcal{B}_W}$ for every finite~${W \subseteq V(D)}$ with~${v,w \in W}$. 
   Now~${\overrightarrow{E}(A,B)}$ is a dicut of~$D$ which does not contain any finite dibond of~$D$, since these properties would already be witnessed for some finite~${W \subseteq V(D)}$. 
   Therefore,~${\overrightarrow{E}(A,B)}$ is an infinite dicut of~$D$ containing only infinite dibonds of~$D$. 
\end{proof}

As noted before, we obtain the following corollary.

\begin{corollary}\label{cor:LY-no_inf-dibond}
    Conjecture~\ref{conj:LY-inf-nested} holds for weakly connected digraphs without infinite dibonds.\qed
\end{corollary}

We close this section with a last special case for which we can verify Conjecture~\ref{conj:LY-inf-nested}.

\begin{corollary}
    \label{cor:LY-rayless}
    Conjecture~\ref{conj:LY-inf-nested} holds for rayless weakly connected digraphs.
\end{corollary}

\begin{proof}
    Let~$D$ be a rayless weakly connected digraph.
    We know by Proposition~\ref{prop:rayless-quotient} that~${D{/}{\equiv_{\mathfrak{B}_\textnormal{fin}}}}$ is rayless as well, and by Proposition~\ref{prop:quotient} that~${D{/}{\equiv_{\mathfrak{B}_\textnormal{fin}}}}$ is weakly connected and finitely diseparable.
    So we obtain from Corollary~\ref{cor:fin-sep+rayless=no_inf-bond} that~${D{/}{\equiv_{\mathfrak{B}_\textnormal{fin}}}}$ has no infinite dibond.
    Now Corollary~\ref{cor:LY-no_inf-dibond} implies that Conjecture~\ref{conj:LY-inf-nested} is true in the digraph~${D{/}{\equiv_{\mathfrak{B}_\textnormal{fin}}}}$.
    Using again that~${D{/}{\equiv_{\mathfrak{B}_\textnormal{fin}}}}$ is finitely diseparable, any nested optimal pair for~${D{/}{\equiv_{\mathfrak{B}_\textnormal{fin}}}}$ directly translates to one for~$D$ by Lemma~\ref{lem:fin-disep}.
    Hence, Conjecture~\ref{conj:LY-inf-nested} is true for~$D$ as well.
\end{proof}

\section{A matching problem about infinite hypergraphs}
\label{sec:hypergraphs}

In this section we discuss how Conjecture~\ref{conj:LY-inf} is related to more general questions about infinite hypergraphs, where the initial one was posed by Aharoni.
We shall give an example, which then negatively answers Aharoni's original question.
However, we leave a modified version as a conjecture which then is still open.
Then we shall strengthen the latter conjecture to obtain a new one, which is closely related to Conjecture~\ref{conj:LY-inf} and the infinite version of Menger's Theorem.
Before we can do this we have to give some definitions and set notation.

Let us fix a hypergraph~${\mathcal{H} = (\mathcal{V}, \mathcal{E})}$. 
We call~$\mathcal{H}$ \textit{simple}, if no hyperedge is contained in another one. 
Given a set~${F \subseteq \mathcal{E}}$, we shall write~${\mathcal{H}[F]}$ for the hypergraph~${(\bigcup F, F)}$ and call it a \textit{subhypergraph of~$\mathcal{H}$}.
Moreover, a subhypergraph~$\mathcal{K}$ of~$\mathcal{H}$ is called \textit{finite}, if there exists some finite~${F \subseteq \mathcal{E}}$ such that~${\mathcal{K} = \mathcal{H}[F]}$. 
Note that~$\mathcal{K}$ might have infinitely many vertices since a hyperedge can contain infinitely many vertices.
We call~$\mathcal{H}$ \textit{locally finite} if each vertex of~$\mathcal{H}$ lies in only finitely many hyperedges.
Furthermore, we say that~$\mathcal{H}$ has \textit{finite character} if no hyperedge of~$\mathcal{H}$ contains infinitely many vertices.

A set of hyperedges~${\mathcal{M} \subseteq \mathcal{E}}$ is called a \textit{matching of~$\mathcal{H}$} if any two hyperedges in~$\mathcal{M}$ are pairwise disjoint. 
A set of vertices~${A \subseteq \mathcal{V}}$ is called a \textit{cover of~$\mathcal{H}$} if every hyperedge of~$\mathcal{H}$ contains a vertex from~$A$.
Now the hypergraph~$\mathcal{H}$ is said to have the \textit{K\H{o}nig property} if a pair~${(\mathcal{M}, A)}$ exists such that the following statements hold:
\begin{enumerate}
    \item $\mathcal{M}$ is a matching of~$\mathcal{H}$.
    \item $A$ is a cover of~$\mathcal{H}$.
    \item ${A \subseteq \bigcup \mathcal{M}}$.
    \item ${\abs{M \cap A} = 1}$ for every~${M \in \mathcal{M}}$. 
\end{enumerate}
We call such a pair~${(\mathcal{M}, A)}$ an \textit{optimal pair for~$\mathcal{H}$}.

Now we are able to state the original problem on infinite hypergraphs posted by Aharoni.

\begin{problem}
    \label{prob:aharoni}
    \cite{aharoni-conj}*{Prob.~6.7}
    Let~$\mathcal{H}$ be a hypergraph and suppose that every finite subhypergraph of~$\mathcal{H}$ has the K\H{o}nig property.
    Does then~$\mathcal{H}$ have the K\H{o}nig property?
\end{problem}

We shall now point out that, in full generality, this problem has a negative answer by stating a certain infinite hypergraph~$\mathcal{H}$.
For this, consider the digraph in Figure~\ref{fig:counterexample}, call it~$D$.
Let~$\mathcal{B}$ denote the set of all dicuts of~$D$.
We now define the hypergraph~$\mathcal{H} = (\mathcal{V}, \mathcal{E})$ by setting~${\mathcal{V} = E(D)}$ and~${\mathcal{E} = \mathcal{B}}$.
As discussed in the introduction,~$\mathcal{H}$ does not have the K\H{o}nig property, since we cannot even find two disjoint hyperedges, but we need infinitely many vertices of~$\mathcal{V}$ to cover all hyperedges.
However, for every non-empty finite subset~$F$ of~$\mathcal{E}$ we can find one vertex of~$\mathcal{V}$ covering all hyperedges of~${(\mathcal{V}, F)}$.

As noticed in the introduction,~$\mathcal{H}$ does not have any finite hyperedges.
This motivates us to modify Problem~\ref{prob:aharoni} to include only hypergraphs of finite character.

\begin{conjecture}
    \label{conj:aharoni_fin_char}
    Let~$\mathcal{H}$ be a hypergraph of finite character and suppose that every finite subhypergraph of~$\mathcal{H}$ has the K\H{o}nig property.
    Then~$\mathcal{H}$ has the K\H{o}nig property.
\end{conjecture}

Variations of Problem~\ref{prob:aharoni}, particularly Conjecture~\ref{conj:aharoni_fin_char}, are very general problems about infinite hypergraphs and probably difficult to answer.
Not much is known about them, not even partial answers.
However, relaxed questions involving fractional matchings and covers have more successfully been studied, see~\cite{aharoni-conj}*{Section~6} for a brief survey on such results.

Now we modify Conjecture~\ref{conj:aharoni_fin_char} even further yielding the following stronger conjecture.

\begin{conjecture}
    \label{conj:aharoni_compact}
    Let~$\mathcal{H} = (\mathcal{V}, \mathcal{E})$ be a hypergraph of finite character and suppose that for every finite~$F \subseteq \mathcal{E}$ there exists some finite set~$F' \subseteq \mathcal{E}$ such that~$F \subseteq F'$ and~$(\mathcal{V}, F')$ has the K\H{o}nig property.
    Then~$\mathcal{H}$ has the K\H{o}nig property.
\end{conjecture}

Although even stronger than Conjecture~\ref{conj:aharoni_fin_char}, this conjecture is very important, because it is closely related to the infinite version of Menger's Theorem and Conjecture~\ref{conj:LY-inf}.
In case Conjecture~\ref{conj:aharoni_compact} is verified, this would not only give another proof of the infinite version of Menger's Theorem proved by Aharoni and Berger~\cite{inf-menger} but also imply Conjecture~\ref{conj:LY-inf}.
The deductions are very similar in both of these cases; namely by defining a suitable auxiliary hypergraph.

For Menger's Theorem where an infinite graph~${G = (V, E)}$ is given as well as vertex sets~${A, B \subseteq V}$, we define an auxiliary hypergraph~${\mathcal{H}_{A, B} = (\mathcal{V}, \mathcal{E})}$ as follows.
The vertex set~$\mathcal{V}$ of~$\mathcal{H}$ consists precisely of those vertices of~$G$ that lie on at least one $A$--$B$~path in~$G$.
Now a subset~${F \subseteq \mathcal{V}}$ forms a hyperedge of~$\mathcal{H}$ if and only if~$F$ is the vertex set of an $A$--$B$~path in~$G$. 
For every finite set~$F$ of~$A$--$B$~paths consider the finite subgraph~$G_F$ induced by the vertex set spanned by the paths in~$F$.
Let~$F'$ be the set of all $A$--$B$~paths in~$G_F$.
Note that~$F'$ is a finite superset of~$F$, for which by Menger's Theorem ${(\mathcal{V}, F')}$ has the K\H{o}nig property. 
Hence, verifying Conjecture~\ref{conj:aharoni_compact} would imply the infinite version of Menger's Theorem.

With respect to the Infinite Lucchesi-Younger Conjecture, consider an infinite weakly connected digraph~${D = (V, E)}$.
We define an auxiliary hypergraph~${\mathcal{H}_{D} = (\mathcal{V}, \mathcal{E})}$ as follows.
We set~${\mathcal{V} = E}$.
Furthermore, a set~${B \subseteq E}$ forms a hyperedge of~$\mathcal{H}_{D}$ if and only if~$B$ is a dibond of~$D$.
Given a finite set~$F$ of dibonds of~$D$ we set~$F'$ to be the minimal finitely corner-closed set of dibonds containing~$F$. 
Note that~$F'$ is still a finite set. 
Now by Theorem~\ref{thm:LY-fin},
respectively~Theorem~\ref{thm:LY-fin-structural},
${(\mathcal{V}, F')}$ has the K\H{o}nig property.
So a positive answer to Conjecture~\ref{conj:aharoni_compact} would imply Conjecture~\ref{conj:LY-inf}. 

Now we conclude this section by translating some results based on compactness arguments of the previous section to yield also verified affirmative answers for special cases of Conjecture~\ref{conj:aharoni_compact}. 
Note first that a corresponding version of Lemma~\ref{lem:fin_parameter_equi} is true for hypergraphs as well:

\begin{lemma}\label{lem:fin_parameter_hyper}
    Let~$\mathcal{H}$ be a hypergraph of finite character.
    Then the following statements are equivalent:
    \begin{enumerate}[label=\textit{(\roman*)}]
        \item \label{item:hyper-1} $\mathcal{H}$ has a finite cover. 
        \item \label{item:hyper-2} The maximal size a matching of~$\mathcal{H}$ can have is finite. 
    \end{enumerate}
\end{lemma}

Using Lemma~\ref{lem:fin_parameter_hyper} we can verify the following special case via the same compactness argument as used for Lemma~\ref{lem:fin-parameter-compactness}. 

\begin{lemma}\label{lem:aharoni_fin_parameter}
    Let~$\mathcal{H}$ be a hypergraph of finite character satisfying the premise of Conjecture~\ref{conj:aharoni_compact}. 
    Furthermore, let~$\mathcal{H}$ satisfy one of the following conditions: 
    \begin{enumerate}[label=\textit{(\roman*)}]
        \item $\mathcal{H}$ has a finite cover. 
        \item There is a finite maximal size a matching of~$\mathcal{H}$ can have.
    \end{enumerate}
    Then~$\mathcal{H}$ has the K\H{o}nig property. 
\end{lemma}

The other result from Section~\ref{sec:cases} we can lift to hypergraphs is Lemma~\ref{L-Y-fin_many_fin_dibonds}.
Again the proof depends on a compactness argument which can immediately be translated into the setting for hypergraphs.

\begin{lemma}\label{lem:aharoni_loc_fin}
     Let~$\mathcal{H}$ be a locally finite hypergraph of finite character satisfying the premise of Conjecture~\ref{conj:aharoni_compact}. 
     Then~$\mathcal{H}$ has the K\H{o}nig property. 
\end{lemma}

\section*{Acknowledgements}
J. Pascal Gollin gratefully acknowledges support by the Institute for Basic Science, No. IBS-R029-C1. 

Karl Heuer was supported by a postdoc fellowship of the German Academic Exchange Service (DAAD) and partly by the European Research Council (ERC) under the European Union's Horizon 2020 research and innovation programme (ERC consolidator grant DISTRUCT, agreement No.\ 648527).

\end{document}